\documentclass[10pt, leqno]{amsart}
\usepackage{amsfonts,amssymb}

\usepackage{amsmath}
\usepackage{amsthm}
\usepackage{amsrefs}
\usepackage{qsymbols}
\usepackage{latexsym}
\usepackage{chngcntr}
\usepackage[noadjust]{cite}
\usepackage{paralist}
\usepackage{xcolor}
\usepackage{esint}

\newtheorem{theorem}{Theorem}[section]

\newtheorem{lemma}[theorem]{Lemma}
\newtheorem{proposition}[theorem]{Proposition}

\theoremstyle{definition}

\newtheorem{remark}[theorem]{Remark}

\newtheorem{assumption}[theorem]{Assumption}

\newcommand{\IR}{\mathbb{R}}
\newcommand{\IC}{\mathbb{C}}
\newcommand{\IN}{\mathbb{N}}

\newcommand{\IE}{\mathbb{E}}
\newcommand{\IP}{\mathbb{P}}

\newcommand{\cX}{\mathcal{X}}
\newcommand{\cY}{\mathcal{Y}}

\newcommand{\cE}{\mathcal{E}}

\newcommand{\cP}{\mathcal{P}}
\newcommand{\cR}{\mathcal{R}}
\newcommand{\cA}{\mathcal{A}}
\newcommand{\cB}{\mathcal{B}}

\newcommand{\cD}{\mathcal{D}}

\renewcommand{\L}{\mathrm{L}}
\newcommand{\C}{\mathrm{C}}
\newcommand{\B}{\mathrm{B}}
\renewcommand{\H}{\mathrm{H}}

\newcommand{\W}{\mathrm{W}}

\newcommand{\e}{\mathrm{e}}

\renewcommand{\d}{\mathrm{d}}
\newcommand{\eps}{\varepsilon}

\newcommand{\divergence}{\operatorname{div}}

\DeclareMathOperator{\supp}{supp}

\DeclareMathOperator{\Id}{Id}

\DeclareMathOperator{\dom}{\mathcal{D}}

\numberwithin{equation}{section} 

\title[The Keller-Segel-Navier-Stokes system in bounded Lipschitz domains]{Strong solutions to the Keller-Segel-Navier-Stokes system in bounded Lipschitz domains}

\author{Matthias Hieber}
\address{TU Darmstadt, Angewandte Analysis, Schlossgartenstraße 764289 Darmstadt, Germany}
\email{hieber@mathematik.tu-darmstadt.de}
\author{Hideo Kozono}
\address{Waseda University, Tokyo, Japan}
\email{kozono@waseda.jp}
\author{Sylvie Monniaux}
\address{Aix Marseille Univ, CNRS, I2M, Marseille, France}
\email{sylvie.monniaux@univ-amu.fr}
\author{Patrick Tolksdorf}
\address{KIT, Fakultät für Mathematik, 76128 Karlsruhe, Germany}
\email{patrick.tolksdorf@kit.edu}

\thanks{Matthias Hieber gratefully acknowledges the support by  the Deutsche  Forschungsgemeinschaft (DFG) through the Research Unit FOR 5528. The research of Hideo Kozono was partially supported by the JSPS Grant-in-Aid for Scientific Research (A)-21H04433.}

\begin{document}
\begin{abstract}
Consider the coupled Keller-Segel-Navier-Stokes or the chemotaxis-consumption-Navier-Stokes 
system in bounded Lipschitz domains for general coupling terms which, e.g., include  
buoyancy forces. It is shown that these systems admit  local strong as well as global 
strong solutions for small data in the setting of critical Besov spaces. Moreover, non-trivial 
equilibria are shown to be exponentially stable. For smoother data, these solutions are shown 
to be globally bounded and to preserve positivity properties.     
The approach presented is based on optimal $\L^q$-regularity properties of the 
Neumann Laplacian and the Stokes operator in bounded Lipschitz domains. 
\end{abstract}
\maketitle

\section{Introduction}

\noindent The theory of reaction diffusion equations as well as the fundamental  equations of fluid dynamics 
have a truly long tradition and a beautiful  history~\cite{KNRY:23}. 
In this paper we are concerned with the coupling of the mechanisms of 
reaction-diffusion equations of chemotaxis type with incompressible visous fluids. 
Our emphasize lies on the situation, where the underyling domains have nonsmooth boundaries, 
more precisely, on domains with Lipschitz boundaries. 

We investigate the  coupled Keller-Segel-Navier-Stokes models  as well as the 
chemotaxis-con\-sumption-Navier-Stokes model in Lipschitz domains.  
There, the chemotaxis  process is dissolved in an incompressible  fluid and the diffusion process 
interacts with the fluid. Problems of this type arise for example  in the investigation of 
the evolution of bacteria, when  besides their chemotactically biased movement towards 
their nutrient, a buoyancy driven effect on the fluid motion is observed, see~\cite{Goldstein}. 
Such chemotaxis and chemotaxis-consumption models coupled to an ambient flow have 
been studied, e.g., in~\cite{PK:92,Goldstein,KNRY:23}.

First analytic results in this direction go back to M.\@ Winkler~\cite{Winkler2016}, 
who proved the existence of global weak solutions to the coupled system. While finite-time 
blowup solutions to the classical Keller-Segel system have been studied in detail by many authors,  
Li and Zhou~\cite{LZ:24} constructed very recently a smooth finite-time blowup solution to the coupled
Keller-Segel-Navier-Stokes system on $\IR^3$. Furthermore, Gong and He~\cite{GH:21} proved 
the global existence of solutions to the two-dimensional coupled parabolic elliptic 
Patlak-Keller-Navier-Stokes system with friction force provided the total mass satisfies 
$M < 8 \pi$.   
For surveys of results concerning the Keller-Segel system and the chemotaxis-consumption model 
itself,  we refer to the articles  of Lankeit and Winkler~\cite{LW:20, LW:23}.  

As written above, our emphasis lies on the investigation of the coupled Keller-Segel-Navier-Stokes 
system in Lipschitz domains. The analysis of elliptic operators and of the Stokes operator 
acting in Lipschitz domains has seen many beautiful  developments in the last years; we refer 
here, e.g., to~\cite{Fabes_Mendez_Mitrea_1998, Jerison_Kenig_1989, Shen2012,
Kunstmann_Weis,Tolksdorf,Gabel_Tolksdorf,Tolksdorf_convex,Brown_Shen_1995,Geng_Shen}. 
In this paper we elaborate these developments further to investigate coupled 
Keller-Segel-Navier-Stokes models as well as the chemotaxis-consumption-Navier-Stokes model 
in Lipschitz domains.

For results concerning the Keller-Segel system itself in Lipschitz domains we refer to the work 
of Horstmann, Meinlschmidt and Rehberg~\cite{HMR:18}. They  proved 
that this system, regarded in Lipschitz domains, possesses a unique local strong solution. A corresponding result in bounded convex domains is due to Hieber, Kress and Stinner~\cite{HKS:21}.
For the special situation of circular sectors in $\IR^2$ of angles $\theta \in (0 , \pi / 2]$, Fuest and Lankeit~\cite{FL:23} recently 
showed that the critical mass for the blow-up dichotomy for the full parabolic Keller-Segel system is given by $4 \theta$ and hence depends directly on the Lipschitz geometry. This has to be seen in contrast to the fact that in arbitrary bounded and smooth subsets of $\IR^2$ global classical solutions always exist for masses less than $4 \pi$, see~\cite{Nagai_Senba_Yoshida}. 

This shows that the Lipschitz geometry indeed plays a special role in the study of blow-up phenomena and provides a further motivation to study these coupled systems in general Lipschitz domains.

We are now describing our system in detail. Let $\Omega \subset \IR^n$, where $n=2$ or $n=3$, be 
a bounded Lipschitz domain  and let  $0 < T \leq \infty$.  
The dynamics of the population density $u : (0 , T) \times \Omega \to \IR$, the concentration 
of the chemical attractant $v : (0 , T) \times \Omega \to \IR$, the fluid 
velocity $w : (0 , T) \times \Omega \to \IR^n$ and the pressure 
$\pi : (0 , T) \times \Omega \to \IR$ are described by the following 
coupled systems of partial differential equations
\begin{align}\tag{KSNS} \label{Eq: KSNS}
 \left\{ \begin{aligned}
  \partial_t u - \Delta u + w \cdot \nabla u &= - \divergence(u \nabla v) &&
  \text{in } (0 , T) \times \Omega, \\
  \partial_t v - \Delta v + v + w \cdot \nabla v &= u && \text{in } 
  (0 , T) \times \Omega, \\
  \partial_t w - \Delta w + (w \cdot \nabla) w + \nabla \pi &= u f && 
  \text{in } (0 , T) \times \Omega, \\
  \divergence(w) &= 0 && \text{in } (0 , T) \times \Omega.
 \end{aligned} \right.
\end{align}
The function $f : \Omega \to \IR^n$ is supposed to be initially known. Of particular interest 
for $n=3$ is the function $f(x)=e_3$, where $e_3= -(0,0,1)^{\top}$. The system is then the 
Keller-Segel-Navier-Stokes system with buoyancy and describes the cell's reaction force to the 
buoyancy exerted by the fluid. 
The blowup scenario for the associated Patlak-Segel-Navier-Stokes system is very interesting. 
Indeed, a smooth finite-time blowup solution for the Patlak-Segel-Navier-Stokes system in three 
dimensions with buoyancy was constructed by Li and Zhou~\cite{LZ:24}. This means that the 
Navier-Stokes equations cannot prevent the coupled system from finite-time blowup. 
On the other hand, Hu, Kiselev and Yao~\cite{HKY:23} found a mechanism of suppression of 
blowup by buoyancy and proved global existence for smooth data for a Keller-Segel system coupled 
with a fluid flow adhering to Darcy's law for porous media via buoyancy forces. 

Note that our approach allows to consider arbritray coupling functions $f \in \L^n(\Omega;\IR^n)$. 
It remains an open problem to decide whether the above blowup and suppression from blowup 
results remain valid for the Keller-Segel-Navier-Stokes and chemotaxis-consumption-Navier-Stokes models
under consideration. 

This system \eqref{Eq: KSNS} is complemented by boundary conditions of the form
\begin{align*} \tag{BC}
 \left\{ \begin{aligned}
  \partial_{\nu} u = \partial_{\nu} v &= 0 && \text{on } 
  (0 , T) \times \partial \Omega, \\
  w &= 0 && \text{on } (0 , T) \times \partial \Omega \\ 
  \end{aligned}\right.
\end{align*}
and initial data 
\begin{equation} \tag{IC}
  u(0) = u_0, \quad  
  v(0) = v_0,  \quad 
  w(0) = w_0  \quad \text{in } \Omega.
\end{equation}
Neglecting the lower-order term $v$ in the second equation 
of~\eqref{Eq: KSNS}, the system has the following scaling invariance:
\begin{align*}
 u_{\lambda} (t , x) &:= \lambda^2 u(\lambda^2 t , \lambda x), 
 \quad v_{\lambda} (t , x):= v(\lambda^2 t , \lambda x), \\
 w_{\lambda} (t , x) &:= \lambda w (\lambda^2 t , \lambda x), 
 \quad \pi_{\lambda} (t , x) := \lambda^2 \pi (\lambda^2 t , \lambda x),
\end{align*}
whenever $f$ is replaced by $f_\lambda (x) := \lambda f (\lambda x)$. 

The aim of this paper is twofold: first we  develop general methods for equations of  
Keller-Segel-Navier-Stokes type in nonsmooth domains in order to establish strong 
well-posedness results for  scaling invariant  critical spaces and secondly we 
study stability of equilibria.   More precisely, we prove local strong well-posedness, 
Theorem~\ref{thm:localex}, as well as global strong well-posedness of 
system~\eqref{Eq: KSNS} for small data, along with the  stability of non-trivial 
equilibria in a critical functional setting, Theorem~\ref{Thm: Existence/Stability}. 
In particular, our framework allows to solve the system for initial data in 
critical Besov spaces. For the case of the uncoupled Navier-Stokes equations on 
$\IR^n$, solutions in critical Besov spaces $\B^{n/q-1}_{q,p}(\IR^n)$ were obtained first by  
Cannone~\cite{Cannone}. In our two well-posedness results we extend his result to the coupled 
Keller-Segel-Navier-Stokes  system in arbitrary bounded Lipschitz domains.    
Moreover, for smooth enough initial data we show that the solutions are 
bounded globally in space and time and that the population density remains 
positive for positive initial data, see Theorem~\ref{Thm: Boundedness and positivity}. 

Both well-posedness results presented in Section 2 are based on the maximal regularity 
approach, see, e.g.,~\cite{CL-93, Ama93, PS2016}. In order to treat the case of critical spaces 
we use the technique of time-weighted maximal regularity developed by Pr\"uss, Simonett 
and Wilke in~\cite{PSW18}. An essential ingredient of our approach is the knowledge of 
optimal regularity properties of the linearization of~\eqref{Eq: KSNS}. 

This  is indeed a very delicate matter, since the underlying domains are only Lipschitz.    
In Lipschitz domains, characterizations of domains of the Neumann Laplacian 
and the Stokes operator pose major problems. Note that, in general, it is not true 
that, even in $\L^2 (\Omega)$, the domain of the Neumann Laplacian is $\H^2 (\Omega)$, 
see, e.g., Jerison and Kenig~\cite[Prop.~1.4]{Jerison_Kenig_1995}. 
Instead, one may establish only embeddings into fractional Sobolev spaces, 
see Fabes, Mendez and Mitrea~\cite[Thm.~9.2]{Fabes_Mendez_Mitrea_1998} for the 
Neumann Laplacian as well as the results due to Brown and Shen~\cite[Thm.~2.12]{Brown_Shen_1995} 
and Mitrea and Wright~\cite[Thm.~10.6.2]{Mitrea_Wright} for the Stokes operator. 

However, optimal descriptions of the domains of the corresponding  square roots in 
terms of first order Sobolev spaces were proved by Jerison and 
Kenig~\cite[Thm.~1]{Jerison_Kenig_1989}, as well as by  
Gabel and Tolksdorf~\cite[Thm.~1.3]{Gabel_Tolksdorf} and~\cite[Thm.~1.1]{Tolksdorf}. 
We will take advantage of these optimal characterizations to achieve results within  
critical scaling invariant spaces.

Our method applies also to the chemotaxis-consumption-Navier-Stokes model, a modified 
version of~\eqref{Eq: KSNS} which was shown to be  relevant to 
describe chemotactic processes in a fluid~\cite{Goldstein}. For analytical results, 
see \cite{Winkler2016,LW:23}.  
This system is given by
\begin{align*} 
\tag{CCNS} \label{Eq: KSNS-variant}
 \left\{ \begin{aligned}
  \partial_t u - \Delta u + w \cdot \nabla u &= - \divergence(u \nabla v) 
  && \text{in } (0 , T) \times \Omega \\
  \partial_t v - \Delta v + w \cdot \nabla v &= - u v && \text{in } (0 , T) \times \Omega \\
  \partial_t w - \Delta w + (w \cdot \nabla) w + \nabla \pi &= u f && \text{in } (0 , T) \times \Omega \\
  \divergence(w) &= 0 && \text{in } (0 , T) \times \Omega
 \end{aligned} \right.
\end{align*}
complemented with the same initial and boundary conditions as above. 

Let us note that all results stated for the system~\eqref{Eq: KSNS} are valid also 
for~\eqref{Eq: KSNS-variant} as well, see Theorem~\ref{Thm: Variant}. The approach presented  
is hence rather robust and can be adjusted to further variants without difficulties.  

This paper is organized as follows. In Section~\ref{Sec: Preliminaries and Main Results} 
we collect well-known results for operators in Lipschitz domains, introduce  the functional 
setting and state our main results, which proofs can be found in 
Section~\ref{Sec: Proof of Local solvability} (local well-posedness), 
Section~\ref{Sec: Global existence for small initial data} 
(global well-posedness and stability) and Sections~\ref{Sec: Positivity} 
and~\ref{Sec: Regularity and boundedness} for positivity and boundedness. 
In the final Section~\ref{Sec: A variant of (KSNS)} we present the treatment 
of~\eqref{Eq: KSNS-variant}.

\section{Preliminaries and Main Results}
\label{Sec: Preliminaries and Main Results}
\noindent Let $\Omega \subset \IR^n$ be a bounded Lipschitz domain and $n = 2$ or $n = 3$. 
We start by rewriting the Keller-Segel-Navier-Stokes system~\eqref{Eq: KSNS} 
as a semilinear evolution equation of the form
\begin{equation}
\label{Eq: abstract system}
 \left\{ \begin{aligned}
 \partial_t U + \cA U &= \Phi(U , U) && \text{in } (0 , T), \\
 U(0) &= U_0
 \end{aligned} \right.
\end{equation}
where $U$ is defined as the triplet $U = (u , v , w)$ and $U_0 = (u_0 , v_0 , w_0)$. 
We study this system on the space 
\begin{align*}
 X_0 := \W^{- 1 , q}_0 (\Omega) \times \L^{q^*} (\Omega) \times \L^q_{\sigma} (\Omega) \quad \text{or on} \quad \cX_0 := \W^{-1 , q}_{\mathrm{av}} (\Omega) \times \L^{q^{*}} (\Omega) \times \L^q_{\sigma} (\Omega),
\end{align*}
where $q^*$ denotes the Sobolev exponent of $\W^{1 , q} (\Omega)$, i.e.,
\[
 \frac{1}{q} - \frac{1}{n} = \frac{1}{q^*},
\]
and $\W^{-1 , q}_0 (\Omega)$  and $\W^{-1 , q}_{\mathrm{av}} (\Omega)$ denote the dual space of $\W^{1 , q^{\prime}} (\Omega)$ and $\W^{1 , q^{\prime}} (\Omega) \cap \L^{q^{\prime}}_{\mathrm{av}} (\Omega)$, respectively, with $1/q + 1/q^{\prime} = 1$. Moreover, $\L^q_{\sigma} (\Omega)$ denotes, as usual, 
the space of solenoidal $\L^q$-vector fields with vanishing normal component and $\L^r_{\mathrm{av}} (\Omega)$ 
the mean value free subspace of $\L^r (\Omega)$ for $1 < r < \infty$. The
operator $\cA$ is given by
\begin{align}
\label{Eq: Operator matrix}
 \cA :=
    \begin{pmatrix}
     - \Delta & 0 & 0 \\ 
     - \Id & - \Delta_{q^*} + \Id & 0 \\
     - \IP [f \boldsymbol{\cdot}] & 0 & A
    \end{pmatrix}
\end{align}
and equipped with the domain
\[
 X_1 := \W^{1 , q} (\Omega) \times \dom(\Delta_{q^*}) \times \dom(A) 
 \quad \text{or} \quad 
 \cX_1 := [\W^{1 , q} (\Omega) \cap \L^q_{\mathrm{av}} (\Omega)] \times 
 \dom( \Delta_{q^{*}}) \times \dom(A).
\]
Here, $\Delta$ denotes the Neumann-Laplacian defined on $\W^{-1 , q}_0 (\Omega)$ 
or on $\W^{- 1 , q}_{\mathrm{av}} (\Omega)$, 
respectively, and $\Delta_r$ the Neumann-Laplacian on $\L^r (\Omega)$ for $1 < r < \infty$.
Furthermore, $A$ denotes the Stokes operator on $\L^q_{\sigma} (\Omega)$ and 
$\IP$ the Helmholtz projection. For $q$ satisfying
\begin{align}
\label{Eq: Lipschitz condition on q}
    \Big\lvert \frac{1}{q} - \frac{1}{2} \Big\rvert < \frac{1}{2 n} + \eps,
\end{align}
where $\eps > 0$ is a number depending on the Lipschitz geometry, it is 
known, see~\cite[Cor.~1.2]{Shen2012} and~\cite[Cor.~1.2]{Gabel_Tolksdorf}, that the Stokes operator generates a bounded analytic semigroup 
on $\L^q_{\sigma} (\Omega)$ and the Helmholtz projection is a bounded operator 
on $\L^q (\Omega ; \IC^n)$, see~\cite[Thm.~11.1]{Fabes_Mendez_Mitrea_1998} and~\cite[Thm.~1.2]{DMitrea}. Observe, that the condition $f \in \L^n (\Omega ; \IR^n)$ implies that $u f \in \L^q (\Omega ; \IR^n)$ so that $\IP [f u]$ is a well-defined element in $\L^q_{\sigma} (\Omega)$. 

We continue by introducing time weighted Banach space valued function spaces. 
Let $1 < p < \infty$, $0 < T \leq \infty$ and $\mu \in (1 / p , 1]$ and let  
$\L^p_{\mu} (0 , T ; X)$ denote the space defined by
\[
  \L^p_{\mu} (0 , T ; X) := \{ u : (0 , T) \to X \text{ measurable} : 
  t \mapsto t^{1 - \mu} u (t) \in \L^p (0 , T ; X) \}
\]
endowed with the norm
\[
 \| u \|_{\L^p_{\mu} (0 , T ; X)} 
 := \bigg( \int_0^T \| t^{1 - \mu} u(t) \|_X^p \, \d t \bigg)^{\frac{1}{p}}.
\] 
The associated weighted Sobolev spaces are defined by
\[
 \H^{1 , p}_{\mu} (0 , T ; X) 
 := \{ u \in \L^p_{\mu} (0 , T ; X) : u^{\prime} \in \L^p_{\mu} (0 , T ; X) \}
\]
equipped with the canonical norm. 

Given a generator $L$ of an analytic semigroup on $X$, we say that $L$ admits maximal 
$\L^p_{\mu}$-regularity on $(0 , T)$, $0 < T \leq \infty$ if for any 
$g \in \L^p_{\mu} (0 , T ; X)$ and $u_0 \in (X , \dom(L))_{\mu - 1 / p , p}$ 
there exists a unique 
\[
 u \in \H^{1 , p}_{\mu} (0 , T ; X) \cap \L^p_{\mu} (0 , T ; \dom(L)) =: 
 \IE_{p , \mu}^T (X , \dom(L))
\]
satisfying 
\[
u^{\prime}+Lu=g, \ u(0)=u_0.
\]
The closed graph theorem implies the existence of some constant $C > 0$ such that
\begin{align}
\label{Eq: Maximal regularity estimate}
 \| u \|_{\IE_{p , \mu}^T (X , \dom(L))} 
 \leq C \big( \|g\|_{\L^p_{\mu}(0,T;X)} + \| u_0 \|_{(X , \dom(L))_{\mu - 1 / p , p}}).
\end{align}
If $L$ has maximal $\L^p_{\mu}$-regularity on $(0 , T_0)$, $0 < T_0 \leq \infty$, 
then $L$ has maximal $\L^p_{\mu}$-regularity on $(0 , T)$ for every $0 < T < T_0$ 
and the constant $C$ can be chosen independently of $T$. Recall that $L$ has maximal 
$\L^p_{\mu}$-regularity if and only if it has maximal 
$\L^p$-regularity~\cite[Thm.~3.5.4]{PS2016} (i.e., $\mu = 1$). \par
Observe that the outcome of the real interpolation space in~\eqref{Eq: Maximal regularity estimate} is often a Besov space. For a brief introduction of all necessary Bessel potential and Besov spaces we refer the reader to the appendix. 
We are now in the position to formulate our first main result on local existence of a unique, strong solution to \eqref{Eq: KSNS}.   

\begin{theorem}[Local existence]\label{thm:localex}
Let $\Omega \subset \IR^n$ be a bounded  Lipschitz domain and $n=2$ or $n=3$. 
Then there exists $\varepsilon= \varepsilon(\Omega) >0$, depending on the 
Lipschitz geometry of $\Omega$, such that for all $p,q \in (1,\infty)$ satisfying
\[
q<n \quad \mbox{and} \quad \frac{n}{q} + \frac{2}{p} \leq 3 \quad \mbox{and} \quad 
\frac{n+1}{2n} - \varepsilon < \frac{1}{q} < 
\frac{n+1}{2n} + \varepsilon
\]
and all 
\[
(u_0,v_0,w_0) \in \B^{n/q-2}_{q , p , 0}(\Omega) \times \B^{n/q^*}_{q^* , p}(\Omega)
 \times \B^{n/q-1}_{q ,p, 0,\sigma}(\Omega) 
\]
and all $f \in \L^n (\Omega;\IR^n)$ there exists $T>0$ such that the Keller-Segel-Navier-Stokes system~\eqref{Eq: KSNS} 
admits a unique strong solution $(u , v , w)$ satisfying 
\begin{align*}
(u,v,w) \in \IE_{p,\mu}^T(X_0 , X_1),
\quad \mbox{where} \quad \mu = \frac{n}{2q} + \frac{1}{p} - \frac{1}{2}.
\end{align*}
\end{theorem}

Next, we formulate our global existence and stability result. We start by describing the set of stationary solutions for which we show exponential stability. For this purpose, we define the mean value of a distribution $u_0 \in \B^{n/q-2}_{q , p , 0}(\Omega)$ as
\[
    u_{0 , \Omega} := \frac{1}{\lvert \Omega \rvert} \langle u_0 , 1 \rangle_{\B^{n / q - 2}_{q , p , 0} , \B^{2 - n / q}_{q^{\prime} , p^{\prime}}}.
\]
Now, a stationary solution to~\eqref{Eq: KSNS} is given by 
\[
 (u_s , v_s , w_s)^{\top} \quad \text{where} \quad u_s = v_s = u_{0 , \Omega}
\]
and $w_s$ satisfies the stationary Navier-Stokes system
\begin{align*}
 \left\{
\begin{aligned}
   - \Delta w_s + (w_s \cdot \nabla) w_s + \nabla \pi_s 
   &= f u_s && \text{in } \Omega, \\
   \divergence{w_s} &= 0 && \text{in } \Omega, \\
   w_s &= 0 && \text{on } \partial \Omega.
\end{aligned} \right.
\end{align*}
Observe that the existence of a stationary weak solution is classical whenever 
$f \in \W^{-1 , 2} (\Omega ; \IR^n)$. In our global existence and stability 
result, the following smallness condition on the product $\lvert u_s \rvert f$ will be crucial.

\begin{assumption}
\label{Ass: Forcing}
For some $\delta > 0$, the data $u_0 \in \B^{n / q - 2}_{q , p , 0} (\Omega)$ 
and $f \in \L^n (\Omega ; \IR^n)$ satisfy 
\begin{enumerate}
 \item 
 $\lvert u_s \rvert \| f \|_{\W^{-1 , 2} (\Omega)} \leq \delta$ in the case $n = 2$;
 \item 
 $\lvert u_s \rvert \| f \|_{\L^{3/2} (\Omega)} \leq \delta$ in the case $n = 3$.
\end{enumerate} 
\end{assumption}

\begin{theorem}[Global existence and stability]
\label{Thm: Existence/Stability}
Let $\Omega \subset \IR^n$ be a bounded  Lipschitz domain and $n=2$ or $n=3$. 
Then there exists $\varepsilon= \varepsilon(\Omega) >0$, depending on the Lipschitz 
geometry of $\Omega$ with the following property. 
For every $p,q \in (1,\infty)$ satisfying
\begin{equation}
\label{Eq: Global p and q}
q<n \quad \mbox{and} \quad \frac{n}{q} + \frac{2}{p} \leq 3 \quad \mbox{and} \quad 
\frac{n+1}{2n} - \varepsilon  < \frac{1}{q} < 
\min\Big(\frac{n+1}{2n} + \varepsilon , \frac{2}{n} \Big)
\end{equation}
there exist $\delta, \lambda , C > 0$ such that for all $0 < \kappa < C$ and all 
\[
(u_0,v_0,w_0) \in \B^{n/q-2}_{q , p , 0}(\Omega) \times \B^{n/q^*}_{q^* , p}(\Omega)
 \times \B^{n/q-1}_{q ,p, 0,\sigma}(\Omega) 
\]
with
\[
 \| u_0 - u_s \|_{\B^{n / q - 2}_{q , p , 0} (\Omega)} + \| v_0 - v_s \|_{\B^{n/q^*}_{q^* , p}(\Omega)}
 + \| w_0 - w_s \|_{\B^{n/q-1}_{q ,p, 0,\sigma}(\Omega)} \leq \kappa
\]
and all $f \in \L^n (\Omega;\IR^n)$ satisfying Assumption~\ref{Ass: Forcing} for this 
$\delta$, there exists a unique global strong solution $(u , v , w)$ 
to~\eqref{Eq: KSNS} satisfying
\[
 \| t \mapsto \e^{\lambda t} [ (u (t) , v (t) , w (t)) - (u_s , v_s , w_s) ] 
 \|_{\IE_{p , \mu}^{\infty} (\cX_0 , \cX_1)} \leq 2 \kappa, \quad \mbox{where} \quad \mu = \frac{n}{2q} + \frac{1}{p} - \frac{1}{2}.
\]
In particular,
\[
 \e^{\lambda t} \bigl\| (u (t) , v (t) , w (t)) - (u_s , v_s , w_s) 
 \bigr\|_{(\cX_0 , \cX_1)_{\mu - 1 / p , p}} \xrightarrow[t\to \infty]{} 0.
\]
\end{theorem}

For the boundedness and positivity of solutions we require slightly more regularity 
conditions on the data. Therefore, we introduce
\[
 Y_0 := \L^q(\Omega)\times\L^{q^{*}}(\Omega)\times\L^q_{\sigma}(\Omega) \quad \text{and} \quad \cY_0 := \L^q_{\mathrm{av}}(\Omega)\times\L^{q^{*}}(\Omega)\times\L^q_{\sigma}(\Omega)
\]
and
\[
    Y_1 := \dom(\Delta_q) \times \dom(\Delta_{q^*}) \times \dom(A) \quad \text{and} \quad \cY_1 := \dom(\Delta_q) \cap \L^q_{\mathrm{av}} (\Omega) \times \dom(\Delta_{q^*}) \times \dom(A).
\]
and let $\IE_{r , 1}^{\infty} (\cY_0 , \cY_1)$ be defined as before.

\begin{theorem}[Boundedness and positivity]
\label{Thm: Boundedness and positivity}
In the setting of Theorem~\ref{Thm: Existence/Stability} there exists $\lambda, \delta , C > 0$ and 
$r_0 > 1$ such that for all $0 < \kappa < C$, all $r > r_0$ and all
\[
 (u_0 , v_0 , w_0) \in \big( Y_0 , Y_1 \big)_{1 - 1 / r , r} \quad \text{with} 
 \quad \| (u_0 - u_s , v_0 - v_s , w_0 - w_s) \|_{( \cY_0 , \cY_1 )_{1 - 1 / r , r}} \leq \kappa
\]
the unique solution constructed in Theorem~\ref{Thm: Existence/Stability} satisfies
\[
 \| t \mapsto \e^{\lambda t} [ (u (t) , v (t) , w (t)) - (u_s , v_s , w_s) ] 
 \|_{\IE_{r , 1}^{\infty} (\cY_0 , \cY_1)} \leq 2 \kappa.
\]
Moreover, it is globally bounded in space and time, i.e.,
\[
 (u , v , w) \in \L^{\infty} (0 , \infty ; \L^{\infty} (\Omega) 
 \times \L^{\infty} (\Omega) \times \L^{\infty} (\Omega ; \IR^n)),
\]
and, if $u_0 \geq 0$ a.e., then $u(t) \geq 0$ for all a.e.\@ $t > 0$ and a.e.\@ $x \in \Omega$.
\end{theorem}

The proofs of our main results rely on optimal $\L_{\mu}^p$-regularity properties 
of the linear part of~\eqref{Eq: KSNS}. If $\Omega$ would have smooth boundaries, the 
result would follow from rather classical results on the Laplacian and the Stokes operator. 
This is, however, not the case for Lipschitz domains as the $\dom(\cA)$ is in general unknown. 
To counteract this problem, we use optimal descriptions of domains of square roots of our 
linear operators at hand.

The domains of the square roots of $- \Delta_q$ as well as $A$ can be characterized 
by first-order Sobolev spaces, i.e.,
\begin{align}
\label{Eq: Square roots}
 \dom((- \Delta_q)^{1/2}) = \W^{1 , q} (\Omega) \quad \text{and} 
 \quad \dom(A^{1 / 2}) = \W^{1 , q}_{0 , \sigma} (\Omega),
\end{align}
whenever $q$ satisfies~\eqref{Eq: Lipschitz condition on q} with some possibly smaller 
$\eps$, see~\cite[Thm.~1]{Jerison_Kenig_1989} for the Laplacian and~\cite[Thm.~1.1]{Tolksdorf} as well as~\cite[Thm.~1.3]{Gabel_Tolksdorf} for the Stokes operator. In the following let $\eps$ be sufficiently small such that all these conditions can be met and let $q$ satisfy the more 
restrictive condition
\begin{align}
\label{Eq: Condition on Lebesgue exponent}
    \frac{n + 1}{2 n} - \frac{\eps}{2} < \frac{1}{q} < \frac{n + 1}{2 n} + \frac{\eps}{2}.
\end{align}
Moreover, there is no harm in assuming that
\begin{align}
\label{Eq: Smallness epsilon}
 \eps \leq \frac{2}{n}.
\end{align}
Indeed, in the case $\eps > \frac{2}{n}$ the conditions~\eqref{Eq: Lipschitz condition on q} and~\eqref{Eq: Condition on Lebesgue exponent} are true for all $1 < q < \infty$. However,~\eqref{Eq: Smallness epsilon} will help us to write certain conditions more concisely. \par

Note that the Neumann Laplacian as well as the Stokes operator admit a 
bounded $\H^{\infty}$-calculus on each components of $X_0$, 
see~\cite[Thm.~1.3]{Egert} and~\cite[Thm.~16]{Kunstmann_Weis}. (For the Neumann Laplacian on the negative spaces $\W^{-1 , q}_0 (\Omega)$ or $\W^{- 1 , q}_{\mathrm{av}} (\Omega)$, respectively, one proceeds similar as in Proposition~\ref{prop:Amaxreg} below: One transfers the boundedness of the $\H^{\infty}$-calculus from $\L^q (\Omega)$ to $\W^{1 , q} (\Omega)$ or $\W^{1 , q} (\Omega) \cap \L^{q}_{\mathrm{av}} (\Omega)$, respectively, by using~\eqref{Eq: Square roots} and then one uses duality to reach the negative space.)
Hence, the complex interpolation space $X_{\alpha}$ between $X_0$ and $X_1$, for $0 < \alpha < 1$, coincides 
with the domains of fractional powers, which means we have that
\[
 X_{\alpha} := \big[ X_0 , X_1 \big]_{\alpha} = \H^{-1 + 2 \alpha , q} (\Omega) 
 \times \dom((- \Delta_{q^*})^{\alpha}) \times \dom(A^{\alpha}).
\]
If $\alpha = 1/2$, then we have
\begin{align}
\label{Eq: X_1/2}
 X_{1/2} = \L^q (\Omega) \times \W^{1 , q^*}(\Omega) \times \W^{1 , q}_{0 , \sigma} (\Omega)
\end{align}
by~\eqref{Eq: Square roots}. Moreover, if $\alpha - 1 / 2 > 0$ is small enough, 
the domains $\dom((- \Delta_{q^*})^{\alpha})$ and $\dom(A^{\alpha})$ optimally embed 
into the Sobolev scale as the following lemma shows. 

\begin{lemma}
\label{Lem: Embeddings fractional power domains}
Let $q \in (1 , \infty)$ be subject to~\eqref{Eq: Condition on Lebesgue exponent}. For $0 < \delta < \frac{n \eps}{4}$ and $0 < \delta^{\prime} \leq \frac{1}{2}$ define
\[
    \alpha = \frac{1}{2} + \delta \quad \text{and} \quad \alpha^{\prime} = \frac{1}{2} + \delta^{\prime}.
\]
Then, 
\begin{align}
\label{Eq: Fractional powers}
 \dom((- \Delta_{q^*})^{\alpha}) \hookrightarrow \W^{1 , r} (\Omega) 
 \quad \text{and} \quad 
 \dom(A^{\alpha^{\prime}}) \hookrightarrow \W^{1 , \kappa}_{0 , \sigma} (\Omega)
\end{align}
where $r$ and $\kappa$ satisfy
\begin{align}
\label{Eq: Definition of kappa}
 \frac{1}{q^*} - \frac{2 \delta}{n} = \frac{1}{r} \quad 
 \text{and} \quad \frac{1}{q} - \frac{2 \delta^{\prime}}{n} = \frac{1}{\kappa}.
\end{align}
\end{lemma}

\begin{proof}
Using~\eqref{Eq: Smallness epsilon}, we remark that for the choices of $\delta$ and $\delta^{\prime}$ we have
\[
    \Big\lvert \frac{1}{\rho} - \frac{1}{2} \Big\rvert < \frac{1}{2 n} + \eps
\]
for the four choices of $\rho$
\[
    \frac{1}{\rho} = \frac{1}{q} - \frac{1}{n} \quad \text{or} \quad 
    \frac{1}{\rho} = \frac{1}{q} - \frac{1}{n} - \frac{2 \delta}{n} \quad \text{or} \quad 
    \frac{1}{\rho} = \frac{1}{q} \quad \text{or} \quad 
    \frac{1}{\rho} = \frac{1}{q} - \frac{2 \delta^{\prime}}{n}.
\]
Writing
\begin{align}
\label{Eq: Splitting beta}
 (\Id - \Delta_{q^*})^{- \alpha} 
 = (\Id - \Delta_{q^*})^{- 1 / 2} (\Id - \Delta_{q^*})^{- \delta}
\end{align}
and noticing that for $g \in \L^{q^*} (\Omega)$ the square root property~\eqref{Eq: Square roots} and a Sobolev embedding implies that
\[
 \| (\Id - \Delta_{q^*})^{- \alpha} g \|_{\W^{1 , r} (\Omega)} \leq C \| (\Id - \Delta_{q^*})^{- \delta} g \|_{\L^r (\Omega)} 
 \leq C \| (\Id - \Delta_{q^*})^{- \delta} g \|_{\H^{2 \delta , q^*} (\Omega)}.
\]
The last term can be estimated by noting that $\dom((- \Delta_{q^*})^{\delta}) = \H^{2 \delta , q^*} (\Omega)$, which follows by~\eqref{Eq: Square roots} and complex interpolation. Thus
\begin{align*}
 \| (\Id - \Delta_{q^*})^{- \alpha} g \|_{\W^{1 , r} (\Omega)} \leq C \| (\Id - \Delta_{q^*})^{- \delta} g \|_{\H^{2 \delta , q^*} (\Omega)} 
 \leq C \| g \|_{\L^{q^*} (\Omega)}.
\end{align*}
We proceed analogously for the Stokes operator, but omit the details.
\end{proof}

The above optimal embeddings in the first-order Sobolev scale will be combined with the 
maximal $\L^p_{\mu}$-regularity property of $\cA$, which is provided by the following 
proposition.

\begin{proposition}
\label{prop:Amaxreg}
Let $1<p<\infty$, $\frac{1}{p}<\mu \leq 1$ and $q \in (1 , \infty)$ be subject to~\eqref{Eq: Condition on Lebesgue exponent}. Then the operator $\cA$ defined in 
\eqref{Eq: Operator matrix} on $X_0$ with domain $X_1$ has maximal $\L^p_{\mu}$-regularity on 
$(0 , T)$ for any $0 < T < \infty$. Moreover, on $\cX_0$ with domain $\cX_1$ it has maximal 
$\L^p_\mu$-regularity on $(0,\infty)$.
\end{proposition}

\begin{proof}
We only consider the global maximal regularity property for $\cA$ on $\cX_0$ with domain $\cX_1$. 
The other case follows a similar reasoning. We remark, that it is classical, that one only 
needs to consider that case of homogeneous initial conditions. Moreover, 
by~\cite[Thm.~3.5.4]{PS2016}, it suffices to prove that $(\cA,\cX_1)$ has maximal 
$\L^p$-regularity, i.e., for $\mu=1$. 

We first give a short argument that shows the maximal $\L^p$-regularity for $-\Delta$ 
on $\W^{-1,q}_{\mathrm{av}}(\Omega)$ relying 
on Amann's inter- and extrapolation scales~\cite[Sec.~6]{amann88}: we regard 
$\Delta_{q^{\prime}}$ as an operator on $\L^{q^{\prime}}_{\mathrm{av}}(\Omega)$, 
so that, by invertibility of $\Delta_{q^{\prime}}$ and heat kernel bounds of the 
corresponding semigroup, it has maximal $\L^p$-regularity on $(0,\infty)$, see~\cite[Thm.~3.1]{Hieber_Pruess}. Moreover, 
since $(- \Delta_{q^{\prime}})^{1/2}$ commutes with the time derivative and with 
$\Delta_{q^{\prime}}$, we conclude that $-\Delta_{q^{\prime}}$ has maximal 
$\L^p$-regularity on $(0,\infty)$ with respect to the ground space 
$\dom((-\Delta_{q^{\prime}})^{1/2}) = \W^{1,q^{\prime}}(\Omega)
\cap\L^{q^{\prime}}_{\mathrm{av}}(\Omega)$. By duality, we infer that $-\Delta$ has 
maximal $\L^p$-regularity on $(0,\infty)$ with respect to the ground space 
$\W^{-1,q}_{\mathrm{av}}(\Omega)$. 

Let $g=(g_1,g_2,g_3)\in \L^p(0,\infty;\cX_0)$. Maximal $\L^p$-regularity for $-\Delta$ on 
$\W^{-1,q}_{\mathrm{av}}(\Omega)$ 
ensures that the problem 
\[
\partial_tu-\Delta u=g_1 \text{ in }(0,\infty)\times\Omega, 
\quad \partial_\nu u=0 \text{ on }(0,\infty)\times\partial\Omega, \quad
u(0,\cdot)=0 \text{ in }\Omega
\]
has a unique solution 
$u\in \W^{1,p}\bigl(0,\infty;\W^{- 1 , q}_{\mathrm{av}} (\Omega)\bigr)
\cap\L^p\bigl(0,\infty;\W^{1,q}(\Omega)\bigr)$. 
This gives that $u\in \L^p\bigl(0,\infty;\L^{q^*}(\Omega)\bigr)$. 
Next, since $\Id - \Delta_{q^*}$ satisfies heat kernel bounds it has maximal 
$\L^p$-regularity on $\L^{q^*}(\Omega)$, see~\cite[Thm.~3.1]{Hieber_Pruess}, and provides a 
solution $v\in \W^{1,p}\bigl(0,\infty;\L^{q^*}(\Omega)\bigr)
\cap \L^p\bigl(0,\infty;\dom(\Delta_{q^*})\bigr)$ of 
\[
\partial_tv+(\Id-\Delta_{q^*}) v =g_2 + u \text{ in }(0,\infty)\times\Omega, 
\quad \partial_\nu v=0 \text{ on }(0,\infty)\times\partial\Omega, \quad
v(0,\cdot)=0 \text{ in }\Omega.
\]
Finally, since by assumption $f\in \L^n(\Omega;{\mathbb{R}}^n)$, we have
that $\IP(uf)\in \L^p\bigl(0,\infty;\L^q_\sigma(\Omega)\bigr)$ and
then maximal $\L^p$-regularity satisfied by the Stokes operator 
$A$ in $\L^q_\sigma(\Omega)$,~\cite[Thm.~16]{Kunstmann_Weis}, ensures that the problem
\[
\partial_t w + A w=g_3 +\IP(uf), \quad w(0,\cdot)=0 \text{ in }\Omega
\]
has a unique solution 
$w\in \W^{1,p}\bigl(0,\infty;\L^q_\sigma(\Omega)\bigr)\cap
\L^p\bigl(0,\infty;\dom(A)\bigr)$.
\end{proof}

\section{Proof of Local solvability}
\label{Sec: Proof of Local solvability}
\noindent In this section we show that system~\eqref{Eq: abstract system} admits a unique 
local strong solution in $\IE_{p , \mu}^T (X_0 , X_1)$. We are going to apply the well-known contraction principle, 
see Lemma~\ref{Lem: fixed point}~\eqref{i} below, stated as in~\cite[Thm.~22.4]{Lem02}, see also~\cite[Lem.~4, Sec.~3.1]{Can05}, and its straightforward variant~\eqref{ii}.

\begin{lemma}[Fixed point theorem]
\label{Lem: fixed point}
Let $E$ be a Banach space and $\Psi:E\times E\to E$ a bounded
bilinear map with norm $C>0$.
\begin{enumerate}
 \item \label{i} For all $0 < \kappa < \frac{1}{4 C}$ and for all $a\in E$ with 
$\|a\|_E \leq \kappa$, there exists a unique $u\in E$ with
\begin{align*}
 \|u\|_E < 2 \kappa \quad \text{such that} \quad u=a+\Psi(u,u).
\end{align*}
    \item \label{ii}
Let $E_0\subset E$ be a Banach space and assume that $\Psi$ takes its values in $E_0$. 
Then for
all $a\in E$ with $\|a\|_{E} \leq \frac{1}{4 C}$, there exists a unique
$v\in E_0$ with 
\begin{align*}
 \|v\|_{E_0} < \frac{1}{4 C} \quad \text{satisfying} \quad v=\Psi(a,a)+\Psi(a,v)+\Psi(v,a)+\Psi(v,v).
\end{align*}
\end{enumerate}
\end{lemma}

\begin{remark}
    A solution $v\in E_0$ given by \eqref{ii} provides a solution $u\in E$ of \eqref{i} 
    with $u=v+a$ in Lemma~\ref{Lem: fixed point}.
\end{remark}

In our situation the bilinear map $\Psi$ is given by $(\partial_t+\cA)^{-1}\Phi$ where
\begin{align}
    \label{eq:bilinearmap}
    \Phi (U_1,U_2) := 
    \begin{pmatrix} \Phi_1(U_1,U_2) \\ \Phi_2 (U_1,U_2) \\ \Phi_3(U_1,U_2) 
    \end{pmatrix} := \begin{pmatrix} 
    - w_1 \cdot \nabla u_2 - \divergence (u_1 \nabla v_2) \\
    - w_1 \cdot \nabla v_2 \\
    - \IP (w_1 \cdot \nabla) w_2 \end{pmatrix}\cdotp
\end{align}
In the following two lemmas we are going to prove its boundedness on 
$\IE_{p , \mu}^T (X_0 , X_1)$.

\begin{lemma} 
\label{lem:bilinearmap}
For every $1 < q < n$ subject to~\eqref{Eq: Condition on Lebesgue exponent} there exist $0 < \delta < \frac{n \eps}{4}$ and $0 < \delta^{\prime} < \frac{1}{2}$ satisfying
\begin{align}
\label{Eq: Sum of deltas}
 \delta + \delta^{\prime} = \frac{1}{2} \Big(  \frac{n}{q} - 1 \Big). 
\end{align}
Moreover, if $\beta := \frac{1}{2} + \delta$ and 
$\beta':= \frac12 + \delta'$, the map $\Phi : (X_{\beta} \cap X_{\beta^{\prime}}) \times (X_{\beta} 
\cap X_{\beta^{\prime}})\to X_0$ given as in~\eqref{eq:bilinearmap} satisfies 
the estimate
\begin{align*}
    \| \Phi(U_1,U_2) \|_{X_0} \leq C \big(\|U_1\|_{X_{\beta}} 
    \|U_2 \|_{X_{\beta^{\prime}}} + \|U_1 \|_{X_{\beta^{\prime}}} \|U_2 \|_{X_{\beta}}\big).
\end{align*}
\end{lemma}

\begin{proof}
First of all, we establish the existence of $\delta$ and $\delta^{\prime}$. We define the continuous and decreasing function
\begin{align*}
 \delta^{\prime} : [0 , \tfrac{n \eps}{4}] \to \IR, \quad \delta^{\prime} (\delta) := \frac{1}{2} \Big( \frac{n}{q} - 1 \Big) - \delta.
\end{align*}
Note that $\delta^{\prime} (0) > 0$ since $q < n$. At the endpoint $\delta = \tfrac{n \eps}{4}$, we use~\eqref{Eq: Condition on Lebesgue exponent} to estimate
\begin{align*}
 \delta^{\prime} (\tfrac{n \eps}{4}) = \frac{1}{2} \Big( \frac{n}{q} - 1 \Big) - \frac{n \eps}{4} < \frac{1}{2} \Big( \frac{n + 1}{2} + \frac{n \eps}{2} - 1 \Big) - \frac{n \eps}{4} = \frac{n - 1}{4} \leq \frac{1}{2} \quad \text{if} \quad n \leq 3.
\end{align*}
Thus, there exists $0 < \delta < \frac{n \eps}{4}$ such that $0 < \delta^{\prime} (\delta) < \frac{1}{2}$. We then define $\delta^{\prime} := \delta^{\prime} (\delta)$. \par
We start by estimating $\Phi_1 (U_1,U_2)$ in $\W^{-1,q}_0(\Omega)$,  
where $U_1 , U_2 \in X_{\beta} \cap X_{\beta^{\prime}}$. To this end, write
\begin{align*}
    \Phi_1 (U_1,U_2) = - w_1\cdot \nabla u_2 - \divergence(u_1\nabla v_2)
    = -\divergence(w_1 u_2)-\divergence(u_1\nabla v_2)
\end{align*}
since $w_1$ is divergence-free.
Consequently, the $\W^{-1 , q}_0 (\Omega)$-norm of $\Phi_1 (U_1,U_2)$ is estimated by
\begin{align*}
    \| \Phi_1 (U_1,U_2) \|_{\W^{-1 , q}_0 (\Omega)} &\leq \|w_1 u_2\|_{\L^q (\Omega)} 
    + \| u_1 \nabla v_2\|_{\L^q (\Omega)} \\
    &\leq \| w_1\|_{\L^r (\Omega)} \| u_2 \|_{\L^s (\Omega)} 
    + \| u_1 \|_{\L^s (\Omega)} \| \nabla v_2\|_{\L^r (\Omega)},
\end{align*}
where $\frac{1}{q} = \frac{1}{r} + \frac{1}{s}$. Notice that 
\[
    u_1 , u_2 \in \H^{-1 + 2 \beta^{\prime} , q} (\Omega) = 
    \H^{2 \delta^{\prime} , q} (\Omega) \hookrightarrow \L^s (\Omega) \quad 
    \text{whenever} \quad \frac{1}{q} - \frac{2 \delta^{\prime}}{n} = \frac{1}{s}.
\]
Moreover, by virtue of~\eqref{Eq: Fractional powers}, we obtain 
\[
    v_1 , v_2 \in \dom((- \Delta_{q^*})^{\beta}) \hookrightarrow 
    \W^{1 , r} (\Omega) \quad \text{whenever} \quad 
    \frac{1}{q^*} - \frac{2 \delta}{n} = \frac{1}{r}.
\]
Finally, employing~\eqref{Eq: Fractional powers} delivers
\begin{align}
\label{Eq: Fractional embedding Stokes}
    w_1 , w_2 \in \dom(A_q^{\beta}) \hookrightarrow \W^{1 , \kappa}_{0 , \sigma} (\Omega) 
    \hookrightarrow \L^r (\Omega) \quad \text{whenever} \quad 
    \frac{1}{\kappa} - \frac{1}{n} = \frac{1}{r},
\end{align}
where $\kappa$ is given in~\eqref{Eq: Definition of kappa}. Putting all the 
conditions on the parameters together reduces to
$\delta + \delta^{\prime} = \frac{1}{2} ( \frac{n}{q} - 1 )$. 
This gives
\[
    \| \Phi_1 (U_1,U_2) \|_{\W^{-1 , q}_0 (\Omega)} \leq 
    C \big( \| U_1\|_{X_{\beta}} \| U_2 \|_{X_{\beta^{\prime}}} + 
    \| U_2 \|_{X_{\beta}} \| U_1\|_{X_{\beta^{\prime}}} \big).
\]
Concerning the estimate of $\Phi_2 (U_1,U_2)$, write
\[
    \Phi_2 (U_1,U_2) = - w_1 \cdot \nabla v_2.
\]
H\"older's inequality delivers
\[
    \| \Phi_2 (U_1,U_2) \|_{\L^{q^*} (\Omega)} \leq 
    \| w_1 \|_{\L^{\tau} (\Omega)} \| \nabla v_2 \|_{\L^r (\Omega)} 
\]
with $r$ is given as above and $\tau$ satisfies  
$\frac{1}{q^*} = \frac{1}{\tau} + \frac{1}{r}$. 
Using that 
\[    
w_1 \in \dom(A_q^{\beta^{\prime}}) \hookrightarrow 
\W^{1 , \kappa^{\prime}}_{0 , \sigma} (\Omega) \hookrightarrow 
\L^{\tilde{\tau}} (\Omega) \quad \text{whenever} \quad 
\frac{1}{\tilde{\tau}} = \frac{1}{\kappa^{\prime}} - \frac{1}{n} 
= \frac{1}{q} - \frac{1 + 2 \delta^{\prime}}{n}
\]
and noticing that the conditions above imply $\tau = \tilde{\tau}$ yields  
\[
    \| \Phi_2 (U_1,U_2) \|_{\L^{q^*} (\Omega)} 
    \leq C \|U_2 \|_{X_{\beta}} \| U_1 \|_{X_{\beta^{\prime}}}.
\]
Next, we estimate $\Phi_3(U_1,U_2)$. Here we obtain 
\[
    \| - \IP (w_1 \cdot \nabla) w_2 \|_{\L^q (\Omega)} 
    \leq C\| w_1\|_{\L^r (\Omega)} \| \nabla w_2 \|_{\L^s (\Omega)}.
\]
Observe that $s = \kappa^{\prime}$ so that if $w_2 \in \dom(A_q^{\beta^{\prime}})$, 
then
\[
    \| \nabla w_2 \|_{\L^s} \leq C \| U_2 \|_{X_{\beta^{\prime}}}
\]
and therefore
\[
\|\Phi_3(U_1,U_2)\|_{L^q}\leq \|U_1 \|_{X_{\beta}} \| U_2\|_{X_{\beta^{\prime}}}. \qedhere
\]
\end{proof}

\begin{lemma}
\label{lem:bilinearestimates}
Let $0 < T \leq \infty$, $p \in (1 , \infty)$, $1 < q < n$ be subject to~\eqref{Eq: Condition on Lebesgue exponent} and $\mu \in (1 / p , 1]$ satisfy
\begin{equation}
\label{eq:q,p,mu}
 \frac{n}{2q}-\frac{1}{2}+\frac{1}{p} \leq \mu \leq 1.
\end{equation}
Then there exists $C > 0$, independent of $T$, 
such that for all $U_1 , U_2 \in \IE_{p , \mu}^T (X_0 , X_1)$ with $U_1|_{t = 0} = 0 = U_2|_{t = 0}$, we have
\begin{equation}
\label{eq:space-time estimate Phi}
    \| \Phi(U_1,U_2) \|_{\L^p (0 , T ; X_0)} 
    \leq C \| U_1 \|_{\IE_{p , \mu}^T (X_0 , X_1)} 
    \| U_2 \|_{\IE_{p , \mu}^T (X_0 , X_1)} .
\end{equation}   
\end{lemma}

\begin{proof}
Let $\mu \in (1/p , 1]$. Then, for $1 < r < \infty$ and
$\frac{1}{r} + \frac{1}{r^{\prime}} = 1$, 
define $\sigma , \sigma^{\prime} \in (\mu , 1]$ as
\begin{align}
\label{Eq: Relation of weights}
    1 - \mu = r (1 - \sigma) = r^{\prime} (1 - \sigma^{\prime}).
\end{align}
With appropriate choices of $\beta , \beta^{\prime} \in (\frac{1}{2} , 1)$, Lemma~\ref{lem:bilinearmap} implies
\begin{align*}
    \| \Phi(U_1,U_2) \|_{\L^p_{\mu} (0 , T ; X_0)} 
    &\leq C \Big( \| \| U_1 \|_{X_{\beta}} 
    \| U_2 \|_{X_{\beta^{\prime}}} \|_{\L^p_{\mu} (0 , T)} 
    + \| \| U_1 \|_{X_{\beta^{\prime}}} 
    \| U_2 \|_{X_{\beta}} \|_{\L^p_{\mu} (0 , T)} \Big) \\
    &\leq C \Big( \| U_1 \|_{\L^{pr}_{\sigma} (0 , T ; X_{\beta})} 
    \| U_2 \|_{\L^{p r^{\prime}}_{\sigma^{\prime}} (0 , T ; X_{\beta^{\prime}})} \\
   &\qquad + \| U_1 \|_{\L_{\sigma^{\prime}}^{p r^{\prime}} ( 0 , T ; X_{\beta^{\prime}})} 
    \| U_2 \|_{\L^{p r}_{\sigma} (0 , T ; X_{\beta})} \Big).
\end{align*}
We proceed by virtue of a weighted Sobolev inequality. Let, for a moment, $T$ be finite. To have an implicit constant that is independent of $T$, we rescale $U_1$ and $U_2$ by defining $V_1 , V_2 \in E^1_{p , \mu}$ via
\begin{align*}
 V_1 (t) := U_1 (t T) \quad \text{and} \quad V_2 (t) := U_2 (t T) \qquad (t \in (0 , 1)). 
\end{align*}
In the following, $C > 0$ denotes a generic constant independent of $T$. The weighted Sobolev embeddings
\[
    \H^{1 - \beta , p}_{\mu} (0 , 1 ; X_{\beta}) 
    \hookrightarrow \L^{p r}_{\sigma} (0 , 1 ; X_{\beta}) 
    \quad \text{and} \quad 
    \H^{1 - \beta^{\prime} , p}_{\mu} (0 , 1 ; X_{\beta^{\prime}}) 
    \hookrightarrow \L^{p r^{\prime}}_{\sigma^{\prime}} (0 , 1 ; X_{\beta^{\prime}})
\]
stated in~\cite[Cor.~1.4]{Meyries-Veraar} are true provided
\begin{align}
\label{Eq: Weighted Sobolev}
    1-\beta-\frac{1}{p}-(1-\mu) \geq -\frac{1}{p r} - (1-\sigma) 
    \quad \mbox{and} \quad 
    1-\beta^{\prime}-\frac{1}{p}-(1-\mu) \geq -\frac{1}{p r^{\prime}}-(1-\sigma^{\prime}).
\end{align}
Observe that this imposes additional constraints on the parameter $r$. Let us show that such a choice is possible. Using~\eqref{Eq: Relation of weights}, the previous two inequalities translate to
\begin{align*}
 1 - \beta - \frac{1}{p} - (1 - \mu) \geq - \frac{1}{r} \Big( 1 + \frac{1}{p} - \mu \Big) \quad \mbox{and} \quad 1 - \beta^{\prime} \geq \frac{1}{r} \Big( 1 + \frac{1}{p} - \mu \Big).
\end{align*}
Consequently, such an $r$ exists if the inequalities
\begin{align}
\label{Eq: Auxiliary condition r}
 1 < r < \infty \quad \text{and} \quad \Big(\beta + \frac{1}{p} - \mu\Big) \Big( 1 + \frac{1}{p} - \mu \Big)^{-1} \leq \frac{1}{r} \leq (1 - \beta^{\prime}) \Big( 1 + \frac{1}{p} - \mu \Big)^{-1}
\end{align}
are simultaneously true. We first verify that the left-hand side of the second condition in~\eqref{Eq: Auxiliary condition r} is indeed smaller than the right-hand side. Using~\eqref{Eq: Sum of deltas} this is equivalent to
\begin{align*}
 \frac{1}{2} \Big( \frac{n}{q} + 1 \Big) + \frac{1}{p} - \mu \leq 1 \quad \Leftrightarrow \quad \frac{n}{2 q} + \frac{1}{p} - \frac{1}{2}  \leq \mu.
\end{align*}
Now, the condition $1 < r < \infty$ can be guaranteed if the left-hand side of the second condition in~\eqref{Eq: Auxiliary condition r} is strictly smaller than 1. But this is true since $\beta < 1$. Thus, an appropriate choice of $r$ can be made in such a way that~\eqref{Eq: Weighted Sobolev} is valid. As a consequence, we can use the weighted Sobolev inequality from~\cite[Cor.~1.4]{Meyries-Veraar} to deduce
\begin{align*}
 \| U_1 \|_{\L^{p r}_{\sigma} (0 , T ; X_{\beta})} = T^{1 - \sigma + \frac{1}{p r}} \| V_1 \|_{\L^{p r}_{\sigma} (0 , 1 ; X_{\beta})} \leq C T^{1 - \sigma + \frac{1}{p r}} \| V_1 \|_{\H^{1 - \beta , p}_{\mu} (0 , 1 ; X_{\beta})}.
\end{align*}
We write the norm of $\H^{1 - \beta , p}_{\mu} (0 , 1 ; X_{\beta})$ in terms of fractional powers as
\begin{align*}
 \| V_1 \|_{\H^{1 - \beta , p}_{\mu} (0 , 1 ; X_{\beta})} \simeq \| \partial_t^{1 - \beta} (1 + \cA^{\beta}) V_1 \|_{\L^{p}_{\mu} (0 , 1 ; X_0)} + \|  (1 + \cA)^{\beta} V_1 \|_{\L^{p}_{\mu} (0 , 1 ; X_0)},
\end{align*}
where $\partial_t$ denotes the time derivative on $\L^p_{\mu} (0 , 1 ; X_0)$ with domain $\dom(\partial_t) = \{ w \in \H^{1 , p}_{\mu} (0 , 1 ; X_0) : w|_{t = 0} = 0 \}$. Since this operator is invertible, we may even write
\begin{align*}
 \| V_1 \|_{\H^{1 - \beta , p}_{\mu} (0 , 1 ; X_{\beta})} \simeq \| \partial_t^{1 - \beta} (1 + \cA)^{\beta} V_1 \|_{\L^{p}_{\mu} (0 , 1 ; X_0)}.
\end{align*}
For a given parameter $\nu > 0$, we apply the Kalton-Weis theorem~\cite[Thm.~4.5.6]{PS2016} to the $\cR$-bounded holomorphic function $f(z) =z^{1 - \beta} (1 + \cA)^{\beta} (\nu z + \nu^{1 - \frac{1}{\beta}} (1 + \cA))^{-1}$ and observe that its $\cR$-bound is independent of $\nu$. Thus, we have
\begin{align*}
 \| U_1 \|_{\L^{p r}_{\sigma} (0 , T ; X_{\beta})} &\leq C T^{1 - \sigma + \frac{1}{p r}} \| \partial_t^{1 - \beta} (1 + \cA)^{\beta} V_1 \|_{\L^{p}_{\mu} (0 , 1 ; X_0)} \\
 &\leq C T^{1 - \sigma + \frac{1}{p r}} \Big( \nu \| \partial_t V_1 \|_{\L^{p}_{\mu} (0 , 1 ; X_0)} + \nu^{1 - \frac{1}{\beta}}\| (1 + \cA) V_1 \|_{\L^{p}_{\mu} (0 , 1 ; X_0)} \Big)
\end{align*}
with a constant $C$ being independent of $T$ and $\nu$. We insert $V_1 (t) = U_1 (t T)$ and obtain
\begin{align*}
 \| U_1 \|_{\L^{p r}_{\sigma} (0 , T ; X_{\beta})} \leq C T^{\mu - \sigma + \frac{1}{p r} - \frac{1}{p}} \Big( \nu T \| \partial_t U_1 \|_{\L^{p}_{\mu} (0 , T ; X_0)} + \nu^{1 - \frac{1}{\beta}}\| (1 + \cA) U_1 \|_{\L^{p}_{\mu} (0 , T ; X_0)} \Big).
\end{align*}
Observe that $\mu - \sigma + \frac{1}{p r} - \frac{1}{p} = \beta - 1$. We choose $\nu = T^{- \beta}$ and deduce
\begin{align*}
 \| U_1 \|_{\L^{p r}_{\sigma} (0 , T ; X_{\beta})} \leq C \Big( \| \partial_t U_1 \|_{\L^{p}_{\mu} (0 , T ; X_0)} + \| (1 + \cA) U_1 \|_{\L^{p}_{\mu} (0 , T ; X_0)} \Big) \leq C \| U_1 \|_{\IE_{p , \mu}^T (X_0 , X_1)}
\end{align*}
for some constant $C$ independent of $T$. Because of this independence, we might even transfer this estimate to the case $T = \infty$. One derives an analogous estimate for $U_2$ and thus arrives at estimate~\eqref{eq:space-time estimate Phi}.
\end{proof}

We are now in the position to apply the contraction principle to establish the local 
strong well-posedness theorem.


\begin{proof}[Proof of Theorem~\ref{thm:localex}]
We start by rewriting the abstract Keller-Segel-Navier-Stokes system \eqref{Eq: abstract system} 
as the fixed point problem
\[
U(t)= \e^{-t{\mathcal{A}}}U_0 +(\partial_t+{\mathcal{A}})^{-1}\Phi(U,U)(t), \quad t\in[0,T].
\]
If $p , q \in (1 , \infty)$ fulfill the conditions of Theorem~\ref{thm:localex}, we have the bilinear estimate from Lemma~\ref{lem:bilinearestimates} at our disposal. Here, we choose $\mu = \frac{n}{2 q} + \frac{1}{p} - \frac{1}{2}$. Note that $\mu > \frac{1}{p}$ is true since $q < n$ and that $\mu \leq 1$ holds by the conditions assumed in the theorem. Thus, by virtue of the contraction mapping principle, Lemma~\ref{Lem: fixed point}~\eqref{ii}, it suffices to prove that 
\[
\|t\mapsto \e^{-t{\mathcal{A}}}U_0\|_{\IE_{p , \mu}^T (X_0 , X_1)}\xrightarrow[T\to 0]{}0,
\]
which is the case whenever $U_0\in(X_0,X_1)_{\mu-1/p,p}$. The characterization
of the latter space in terms of appropriate Besov spaces is delicate since the 
domain of the associated operators are not explicit. By the reiteration 
theorem~\cite[Thm.~2, Sec.~1.10.3]{Triebel} we have
\[
 \big( X_0 , X_1 \big)_{\alpha , p} = \big( X_0 , X_{1/2} \big)_{2 \alpha , p} 
 \quad \text{for} \quad 0 < \alpha < \frac{1}{2}.
\]
Combining~\eqref{Eq: X_1/2} with the interpolation results 
in~\cite[Thm.~2.13]{Triebel_Lipschitz} and Proposition~\ref{Prop: Real interpolation} we arrive at
\[
 \big( X_0 , X_1 \big)_{\alpha , p} 
 = \B^{2 \alpha - 1}_{q , p}(\Omega) \times \B^{2 \alpha}_{q^{*} , p}(\Omega)
 \times \B^{2 \alpha}_{q , p , 0 ,\sigma}(\Omega) \quad 
 \text{for} \quad 0 < \alpha < \frac{1}{2}.
\]
Taking $\alpha = \mu - \frac{1}{p} = \frac{n}{2 q} - \frac{1}{2}$ yields the claim.
\end{proof}

\section{Global existence for small initial data}
\label{Sec: Global existence for small initial data}

\noindent In this section we aim for a global strong existence result for the
system~\eqref{Eq: KSNS} subject to small initial data. Recall, that 
for a fixed initial datum $u_0 \in \B^{n/q-2}_{q , p , 0}(\Omega)$ we 
defined its mean value as
\[
    u_{0 , \Omega} := \frac{1}{\lvert \Omega \rvert} \langle u_0 , 1 
    \rangle_{\B^{n / q - 2}_{q , p , 0} , \B^{2 - n / q}_{q^{\prime} , p^{\prime}}}.
\]
We decompose
\[
    U = \widehat{U} + \begin{pmatrix} u_s \\ v_s \\ w_s \end{pmatrix} 
    \quad \text{with} \quad u_s := u_{0 , \Omega} =: v_s
\]
where $\widehat{U}$ and $w_s$ are solutions to the following equations. 
The function $w_s$ is supposed to solve the stationary Navier-Stokes system given by
\begin{align}
\label{Eq: Stationary Navier-equations}
\left\{
\begin{aligned}
   - \Delta w_s + (w_s \cdot \nabla) w_s + \nabla \pi_s &= f u_s && \text{in } \Omega, \\
   \divergence{w_s} &= 0 && \text{in } \Omega, \\
   w_s &= 0 && \text{on } \partial \Omega.
\end{aligned} \right.
\end{align}
Furthermore, $\widehat{U} = (\widehat{u} , \widehat{v} , \widehat{w})$ is subject to
\begin{align}
\label{Eq: Mean valued free adapted equation}
    \left\{ 
    \begin{aligned}
    \partial_t \widehat{u} - \Delta \widehat{u} + w_s \cdot \nabla \widehat{u} 
    + u_s \Delta \widehat{v} &= \Phi_1(\widehat{U},\widehat{U})  
    && \text{in } (0 , \infty) \times \Omega, \\
    \partial_t \widehat{v} - \widehat{u} - \Delta \widehat{v} + \widehat{v} 
    + w_s \cdot \nabla \widehat{v} &= \Phi_2(\widehat{U},\widehat{U}) 
    && \text{in } (0 , \infty) \times \Omega, \\
    \partial_t \widehat{w} - \Delta \widehat{w} + \nabla \widehat{\pi} 
    - f \widehat{u} + (w_s \cdot \nabla) \widehat{w} + (\widehat{w} \cdot \nabla) w_s 
    &= - (\widehat{w} \cdot \nabla) \widehat{w} 
    && \text{in } (0 , \infty) \times \Omega, \\
   \divergence{\widehat{w}} &= 0 && \text{in } (0 , \infty) \times \Omega, \\
   \partial_{\nu} \widehat{u} = \partial_{\nu} \widehat{v} = 0 \quad \text{and} \quad \widehat{w} &= 0 && \text{on } (0 , \infty) \times \partial \Omega,
    \end{aligned} 
    \right.
\end{align}
where $\Phi$ is the map defined by \eqref{eq:bilinearmap}, where the system is 
complemented by the initial data
\begin{align*}
\left\{
\begin{aligned}
    \widehat{u}_0 &= u_0 - u_s && \text{in } \Omega, \\
    \widehat{v}_0 &= v_0 - v_s  && \text{in } \Omega, \\
    \widehat{w}_0 &= w_0 - w_s  && \text{in } \Omega.
\end{aligned} \right.
\end{align*}

\begin{lemma}
\label{Eq: Solvability stationary Navier-Stokes}
For any $f \in \L^2 (\Omega ; \IR^n)$ there exists a weak solution 
$w_s \in \dom(A)$, where $A$ denotes the Stokes operator on $\L^2_{\sigma} (\Omega)$, 
solving~\eqref{Eq: Stationary Navier-equations}. In particular, 
$w_s \in \L^{\infty} (\Omega ; \IR^n) \cap \W^{1 , n} (\Omega ; \IR^n)$ and
\[
\| w_s \|_{\W^{1 , n} (\Omega)} \leq C \big( \lvert u_s \rvert \| f \|_{\L^2 (\Omega)} 
 + \lvert u_s \rvert^2 \| f \|_{\L^2 (\Omega)}^2 \big).   
\] 
\end{lemma}

\begin{proof}
We start with the case $n = 3$. Let 
$w_s \in \W^{1 , 2}_{0 , \sigma} (\Omega)$ be a solution to~\eqref{Eq: Stationary Navier-equations} constructed by energy methods, see, e.g.,~\cite[Thm. 1.2, Chap.~II]{Temam}. It satisfies
\begin{align}
\label{Eq: Est w_s}
 \| \nabla w_s \|_{\L^2 (\Omega)} \leq \lvert u_s \rvert \| f \|_{\W^{-1 , 2} (\Omega)}.
\end{align}
Using, in addition, a Sobolev embedding, we estimate the nonlinearity as
\[
 \| (w_s \cdot \nabla) w_s \|_{\L^{3 / 2} (\Omega)} 
 \leq \| w_s \|_{\L^6 (\Omega)} \| \nabla w_s \|_{\L^2 (\Omega)} 
 \leq C \| \nabla w_s \|_{\L^2 (\Omega)}^2 
 \leq C \lvert u_s \rvert^2 \| f \|_{\L^2 (\Omega)}^2
\]
and regard it as a right-hand side for the linear Stokes system. Then, the 
regularity result of Brown and Shen~\cite[Thm.~2.12]{Brown_Shen_1995} yields that 
$w_s \in \W^{3/2 , 2} (\Omega) \hookrightarrow \W^{1 , 3} (\Omega)$ and
\begin{align*}
 \| w_s \|_{\W^{1 , 3} (\Omega)} 
 \leq C \| w_s \|_{\W^{3/2 , 2} (\Omega)} &\leq C \| u_s f 
 - (w_s \cdot \nabla) w_s \|_{\L^{3/2} (\Omega)} \\
 &\leq C \big( \lvert u_s \rvert \| f \|_{\L^2 (\Omega)} 
 + \lvert u_s \rvert^2 \| f \|_{\L^2 (\Omega)}^2 \big).
\end{align*}
Now, that $w_s$ has more regularity, we can bootstrap the regularity of the 
nonlinearity and arrive at
\[
 \| (w_s \cdot \nabla) w_s \|_{\L^2 (\Omega)} 
 \leq \| w_s \|_{\L^6 (\Omega)} \| \nabla w_s \|_{\L^3 (\Omega)} 
 \leq C \big( \lvert u_s \rvert \| f \|_{\L^2 (\Omega)} 
 + \lvert u_s \rvert^2 \| f \|_{\L^2 (\Omega)}^2 \big) \lvert u_s \rvert \| f \|_{\L^2 (\Omega)}.
\]
Regarding the nonlinearity again as a right-hand side we obtain, 
since $f \in \L^2 (\Omega ; \IR^n)$, that $w_s$ is in the domain of $A$ on 
$\L^2_{\sigma} (\Omega)$. Since $\dom(A)$ embeds into $\L^{\infty} (\Omega ; \IR^n)$, 
see~\cite[Thm.~2.12]{Brown_Shen_1995}, we conclude the proof of the three-dimensional case. \par
The case $n = 2$ is similar. Here, one may directly use the regularity result of Mitrea and Wright~\cite[Thm.~10.6.2]{Mitrea_Wright} for the embedding into $\L^{\infty} (\Omega ; \IR^2)$ once $(w_s \cdot \nabla) w_s$ is shown to belong to, e.g., $\L^{3/2} (\Omega ; \IR^2)$.
\end{proof}

The system~\eqref{Eq: Mean valued free adapted equation} can be rewritten 
equivalently as an evolution equation
\begin{align}
\label{Eq: Abstract mean value free equation}
\left\{ \begin{aligned}
    \partial_t \widehat{U} + \cA \widehat{U} + \cB_s \widehat{U} 
    &= \Phi(\widehat{U} , \widehat{U}) && \text{for } t \in (0 , \infty), \\
    \widehat{U} (0) &= \widehat{U}_0,
    \end{aligned} \right.
\end{align}
on the space $\cX_0$ defined in Section~\ref{Sec: Preliminaries and Main Results}. 
The operator $\cA$, introduced in~\eqref{Eq: Operator matrix}, 
is now equipped with the domain $\cX_1$ and the operator $\cB_s$ is given by
\[
    \cB_s := \begin{pmatrix}
    w_s \nabla & u_s \Delta & 0 \\
    0 & w_s \nabla & 0 \\
    0 & 0 & \IP [(w_s \nabla) \, \boldsymbol{\cdot} \,] + \IP [( \,\boldsymbol{\cdot}\, \nabla)w_s] 
    \end{pmatrix},
\]
defined on the domain $\dom(\cB_s) := \cX_1$. The nonlinearity 
$\Phi(\widehat{U},\widehat{U})$ is the same as in~\eqref{eq:bilinearmap}.

\begin{lemma}
\label{Lem: Perturbation}
Let $p \in (1 , \infty)$ and $q \in (\frac{n}{2} , \infty)$ be subject to~\eqref{Eq: Condition on Lebesgue exponent}. There exists $\delta > 0$ such that for all $f \in \L^n (\Omega ; \IR^n)$ 
satisfying 
\begin{align*}
 \lvert u_s \rvert \|f\|_{\L^2(\Omega)}+\lvert u_s \rvert \leq \delta
\end{align*}
the operator $\cA+\cB_s$ has maximal $\L^p_{\mu}$-regularity on $(0,\infty)$ 
on the ground space $\cX_0$. In particular, $0 \in \rho(\cA+\cB_s)$.
\end{lemma}

\begin{proof}
Recall that an operator that has maximal $\L^p_{\mu}$-regularity on $(0 , \infty)$ is necessarily invertible. Thus, the invertibility as well as the maximal $\L^p_{\mu}$-regularity of $\cA$ follow from Proposition~\ref{prop:Amaxreg}. 

The maximal regularity property of $\cA + \cB_s$ is deduced 
by a relative bounded perturbation criterion~\cite[Prop.~4.2]{DHP}. 
More precisely, given $\eps > 0$ we verify that
\begin{align}
\label{Eq: Relative bounded}
    \|\cB_s \widehat{U}\|_{\cX_0} \leq \eps \|\cA \widehat{U}\|_{\cX_0} 
    \qquad (\widehat{U} \in \cX_1)
\end{align}
whenever $\delta$ is small enough. In particular, this criterion together with invertibility guarantees the invertibility of $\cA + \cB_s$ once~\eqref{Eq: Relative bounded} is verified. \par
Notice that Lemma~\ref{Eq: Solvability stationary Navier-Stokes} together with the assumption guarantees that for any $1 < r < \infty$
\begin{align*}
 \| w_s \|_{\L^r (\Omega)} \leq C \delta (1 + \delta) \quad \text{and} \quad \| \nabla w_s \|_{\L^n (\Omega)} \leq C \delta (1 + \delta).
\end{align*}
We will frequently use these estimates. For instance, since 
$\dom(\Delta) = \W^{1 , q} (\Omega) \cap \L^q_{\mathrm{av}} (\Omega) \hookrightarrow \L^{q^*}(\Omega)$, we find
\begin{align*}
    \|w_s \cdot \nabla \widehat{u}\|_{\W^{-1,q}_{\mathrm{av}}(\Omega)} 
    = \|\divergence(w_s \widehat{u})\|_{\W^{-1,q}_{\mathrm{av}}(\Omega)} 
    &\leq \|w_s\widehat{u}\|_{\L^q(\Omega)} \\
    &\leq C \|w_s\|_{\L^n(\Omega)} \| \widehat{u} \|_{\L^{q^*}(\Omega)} 
    \leq C \delta (1 + \delta) \| \Delta \widehat{u}\|_{\W^{-1,q}_{\mathrm{av}} (\Omega)}. 
\end{align*}
Similarly, since $\nabla (\Id - \Delta_{q^*})^{-1}$ is bounded on $\L^{q^*} (\Omega) \hookrightarrow \L^q (\Omega)$, we have
\begin{align*}
    \|u_s \Delta \widehat{v}\|_{\W^{-1,q}_{\mathrm{av}}(\Omega)} 
    \leq C \lvert u_s \rvert \| \nabla \widehat{v}\|_{\L^q(\Omega)} 
    &\leq C \delta \| (\Id - \Delta_{q^{*}}) \widehat{v}\|_{\L^{q^*} (\Omega)} \\
    &\leq C \delta \big( \| (\Id - \Delta_{q^{*}}) \widehat{v} - \widehat{u}\|_{\L^{q^*} (\Omega)} +   \| \Delta \widehat{u}\|_{\W^{-1,q}_{\mathrm{av}} (\Omega)} \big).
\end{align*}
Next, let $r > q^*$ satisfy~\eqref{Eq: Lipschitz condition on q}. The property~\eqref{Eq: Square roots} implies $\| \nabla \widehat{v} \|_{\L^r (\Omega)} \leq C \| (\Id - \Delta_{q^*})^{1/2} \widehat{v} \|_{\L^r (\Omega)}$. Moreover, since $q^* > n$ a Sobolev embedding yields $\dom((\Id - \Delta_{q^*})^{1/2}) = \W^{1 , q^*} (\Omega) \hookrightarrow \L^r (\Omega)$. Thus, if $r^{\prime}$ is such that $1 / r^{\prime} + 1 / r = 1 / q^*$, we conclude
\begin{align*}
    \|w_s\cdot\nabla\widehat{v} \|_{\L^{q^*} (\Omega)} 
    &\leq \|w_s\|_{\L^{r^{\prime}}(\Omega)} \|\nabla \widehat{v}\|_{\L^r(\Omega)} \\
    &\leq C \delta (1 + \delta) 
    \|(\Id -\Delta_{q^*})^{1/2}\widehat{v}\|_{\L^r(\Omega)} 
    \\
    &\leq C \delta (1 + \delta) \| (\Id - \Delta_{q^*})\widehat{v}\|_{\L^{q^*}(\Omega)} \\
    &\leq C \delta (1 + \delta) \big( \| (\Id - \Delta_{q^*})\widehat{v} - \widehat{u}\|_{\L^{q^*}(\Omega)} + \| \Delta \widehat{u}\|_{\W^{-1,q}_{\mathrm{av}} (\Omega)}
 \big).
\end{align*}
In order to estimate the remaining term, choose $q<\rho<q^*$ and 
$1/\rho^{\prime} + 1/\rho = 1/q$. Note that, by Lemma~\ref{Lem: Embeddings fractional power domains}, we have that $\dom(A) \hookrightarrow \W^{1 , q^*}_{\sigma} (\Omega)$, so that
\begin{align*}
    \|\IP[(\widehat{w}\cdot \nabla)w_s]\|_{\L^q(\Omega)} 
    &+ \|\IP[(w_s\cdot \nabla) \widehat{w}]\|_{\L^q(\Omega)} \\
    &\leq C \|\widehat{w}\|_{\L^{q^*}(\Omega)} \|\nabla w_s\|_{\L^n(\Omega)} 
    + \|w_s\|_{\L^{\rho^{\prime}}(\Omega)}\|\nabla\widehat{w}\|_{\L^{\rho}(\Omega)}\\
    &\leq C \delta (1 + \delta) \|A \widehat{w}\|_{\L^q(\Omega)} \\
    &\leq C \delta (1 + \delta) \big( \|A \widehat{w} - \IP[f \widehat{u}]\|_{\L^q(\Omega)} + \| f \|_{\L^n (\Omega)} \| \widehat{u} \|_{\L^{q^*} (\Omega)} \big) \\
    &\leq C \delta (1 + \delta) \big( \|A \widehat{w} - \IP[f \widehat{u}]\|_{\L^q(\Omega)} + \| f \|_{\L^n (\Omega)} \| \Delta \widehat{u} \|_{\W^{-1,q}_{\mathrm{av}} (\Omega)} \big).
\end{align*}
Choosing $\delta$ small enough yields~\eqref{Eq: Relative bounded}.
\end{proof}

\begin{remark}
An analysis of the proofs of Lemma~\ref{Eq: Stationary Navier-equations} and~\ref{Lem: Perturbation} shows that the following weaker smallness conditions would suffice:
\begin{align*}
 \lvert u_s \rvert \|f\|_{\W^{-1 , 2}(\Omega)} + \lvert u_s \rvert \leq \delta \quad \text{if } n = 2 \quad \text{or} \quad \lvert u_s \rvert \|f\|_{\L^{3/2}(\Omega)} + \lvert u_s \rvert \leq \delta \quad \text{if } n = 3.
\end{align*}
\end{remark}

\begin{proof}[Proof of Theorem~\ref{Thm: Existence/Stability}]
Let $U_s := (u_s,v_s,w_s)$ be a stationary solution constructed above. 
In the following, we will construct a global-in-time solution $\widehat{U}$ 
of~\eqref{Eq: Abstract mean value free equation}. The solution $U$ 
to~\eqref{Eq: KSNS} will then be given by $U := \widehat{U} + U_s$.

To show exponential stability of the stationary solution $U_s$, note that 
$V := \e^{\omega t} \widehat{U}$, $\omega > 0$, satisfies
\begin{align} 
\label{Eq: Syst V}
\left\{ 
    \begin{aligned}
    \partial_t V + (\cA - \omega \Id) V + \cB_s V &= \e^{-\omega t} \Phi(V,V) 
    && \text{for } t \in (0,\infty), \\
    V (0) &= \widehat{U}_0.
    \end{aligned} 
    \right.
\end{align}
Let $\delta > 0$ be small enough such that Lemma~\ref{Lem: Perturbation} is applicable. Choosing $\omega$ small enough, we see by Lemma~\ref{Lem: Perturbation} that 
$\cA-\omega \Id +\cB_s$ is still invertible and admits maximal $\L^p_{\mu}$-regularity 
on $\cX_0$ globally in time. We write~\eqref{Eq: Syst V} as the fixed point problem
\[
V(t)=\e^{-t({\mathcal{A}}-\omega\Id+\cB_s)} \widehat{U}_0 
+(\partial_t+{\mathcal{A}}-\omega\Id+\cB_s)^{-1}\Phi(V,V)(t), 
\quad t\in [0,\infty).
\]
Using the contraction mapping principle, Lemma~\ref{Lem: fixed point}~\eqref{i}, 
and Lemma~\ref{lem:bilinearestimates} it suffices to prove that
\[
 \|t\mapsto \e^{-t({\mathcal{A}}-\omega\Id+\cB_s)} 
 \widehat{U}_0 \|_{\IE_{p,\mu}^{\infty} (\cX_0,\cX_1)} 
 \quad \text{is small enough.}
\]
This, however, is a direct consequence of~\eqref{Eq: Maximal regularity estimate} 
and the maximal $\L^p_{\mu}$-regularity property of $\cA-\omega\Id+\cB_s$ 
on $(0,\infty)$ whenever $\widehat{U}_0$ is small enough in $(\cX_0 , \cX_1)_{\mu - 1/p , p}$. As in the proof of Theorem~\ref{thm:localex} this interpolation space can be computed as
\begin{align*}
 (\cX_0 , \cX_1)_{\mu - 1/p , p} = \{ h \in \B^{n/q-2}_{q , p , 0}(\Omega) : \langle h , 1 \rangle = 0 \} \times \B^{n/q^*}_{q^* , p}(\Omega)
 \times \B^{n/q-1}_{q ,p, 0,\sigma}(\Omega).
\end{align*}

Since
\[
    \IE_{p,\mu}^{\infty} (\cX_0,\cX_1) \hookrightarrow 
    \mathrm{BUC}([0,\infty);(\cX_0,\cX_1)_{\mu-1/p,p})
\]
we have that $\sup_{0<t<\infty}\|V(t)\|_{(\cX_0,\cX_1)_{\mu-1/p,p}}<\infty$ 
and thus there exists a constant $C>0$ such that for all $t>0$
\[
    \|U(t)-U_s\|_{(\cX_0,\cX_1)_{\mu-1/p,p}} 
    = \e^{-\omega t}\|V(t)\|_{(\cX_0,\cX_1)_{\mu-1/p,p}} 
    \leq C \e^{-\omega t}.
\]
The desired number $\lambda$ could now be chosen as $\lambda := \omega / 2$.
\end{proof}

\section{Positivity}
\label{Sec: Positivity}

\noindent The positivity of the population density will follow from the following auxiliary result for a linear parabolic system subject to homogeneous Neumann boundary conditions. We will say that the normal derivative $\partial_{\nu} v$ of a function $v \in \H^1 (\Omega)$ vanishes on $\partial \Omega$ if
\begin{align*}
 \int_{\Omega} \nabla v \cdot \psi \, \d x = - \int_{\Omega} v \divergence(\psi) \, \d x = 0 \qquad (\psi \in \H^1 (\Omega)).
\end{align*}

\begin{proposition}
\label{Prop: Positivity}
Let $w \in \L^2(0 , T ; \L^2_{\sigma} (\Omega))$ and $v \in \L^2(0 , T ; \H^1 (\Omega))$ 
satisfy $\partial_{\nu} v(t) = 0$ on $\partial \Omega$ for almost every $t \in (0 , T)$. 
Let $u \in \W^{1 , 1} (0 , T ; \L^1 (\Omega)) \cap \L^2 (0 , T ; \H^1 (\Omega))$ 
and $u_0 \in \L^1 (\Omega)$ with $u_0 \geq 0$ a.e.\@ in $\Omega$ satisfy
\begin{align}
\label{Eq: Positivity equation<}
    \left\{ \begin{aligned}
        \partial_t u - \Delta u + w \cdot \nabla u + \divergence(u \nabla v) &= 0 
        && \text{in } (0 , T) \times \Omega, \\
        \partial_{\nu} u &= 0 && \text{on } (0 , T) \times \partial \Omega, \\
        u(0) &= u_0 && \text{in } \Omega.
    \end{aligned} \right.
\end{align}
Then $u (t) \geq 0$ a.e.\@ in $\Omega$ and almost every $t \in (0 , T)$.
\end{proposition}

\begin{proof}
We test the equation by an approximation of $-1$ on $\{ u < 0 \}$ 
and show that $u^-(t) := - \min(0 , u(t))$ vanishes a.e.\@ in $\Omega$. 
For this purpose, let $\eta > 0$ and define
\[
    \varphi_{\eta} (r) := \begin{cases} 0, &\text{if } r \geq 0, \\
     r / \eta, & \text{if } - \eta < r < 0, \\
    -1, &\text{if } r \leq - \eta.\end{cases}
\]
Then
\[
    u(t , \cdot) \varphi_{\eta} (u (t , \cdot)) \to u^{-} (t , \cdot) 
    \quad \text{as} \quad \eta \to 0 \quad \text{in} \quad \L^1 (\Omega)
\]
for almost every $0 < t < T$. In addition, we have
\[
    \nabla \varphi_{\eta} (u) = \varphi_{\eta}^{\prime} (u) \nabla u 
    = \frac{1}{\eta} 1\hspace{-2.2pt}\text{l}_{\{- \eta < u < 0\}} \nabla u.
\]
Now, test the equation~\eqref{Eq: Positivity equation<} by $\varphi_{\eta} (u)$
to obtain for $0 < \tau \leq T$
\begin{align*}
    0 &= \int_0^{\tau}\int_{\Omega}\partial_t u\varphi_{\eta}(u)\,\d x\,\d t 
    +\int_0^{\tau}\int_{\Omega}\nabla u\cdot\varphi_{\eta}^{\prime}(u)\nabla u\,\d x\,\d t
    +\int_0^{\tau}\int_{\Omega}\divergence(wu)\varphi_{\eta}(u)\,\d x\,\d t\\
    &\qquad +\int_0^{\tau}\int_{\Omega}\divergence(u\nabla v)\varphi_{\eta}(u)\,\d x\,\d t.
\end{align*}
Next, we employ the product rule in the first term and an integration by parts in the last two terms and use that $\partial_{\nu} v$ as well as $\nu \cdot w$ vanish on $\partial \Omega$ (since $w(t) \in \L^2_{\sigma} (\Omega)$) and thus arrive at
\begin{align*}
    0 &=\int_0^{\tau}\int_{\Omega}\partial_t\big( u\varphi_{\eta}(u)\big)\,\d x\,\d t - 
    \int_0^{\tau}\int_{\Omega}u\partial_t u\varphi_{\eta}^{\prime}(u)\,\d x\,\d t + 
    \frac{1}{\eta}\int_0^{\tau}\int_{\{-\eta<u<0\}}\lvert\nabla u\rvert^2\,\d x\,\d t\\
    &\qquad -\frac{1}{\eta}\int_0^{\tau}\int_{\{-\eta<u<0\}} 
    u\nabla u\cdot\big(w + \nabla v\big)\,\d x\,\d t.
\end{align*}
Since $u_0 \geq 0$, we have $\varphi_{\eta}(u_0)=0$. Thus, we see that
\[
    \int_0^{\tau}\int_{\Omega}\partial_t \big(u\varphi_{\eta}(u)\big)\,\d x\,\d t 
    =\int_{\Omega} u(\tau,x)\varphi_{\eta} (u(\tau,x))\,\d x.
\]
Rearranging the terms and using that $\lvert u/\eta\rvert\leq 1$ 
on $\{-\eta<u<0\}$ yields that
\begin{align*}
    &\int_{\Omega} u(\tau,x)\varphi_{\eta}(u(\tau,x))\,\d x 
    + \frac{1}{\eta}\int_0^{\tau}\int_{\{-\eta<u<0\}} \lvert\nabla u\rvert^2\,\d x\,\d t \\
    &\leq \int_0^{\tau}\int_{\{-\eta<u<0\}} \lvert\partial_t u\rvert\,\d x\,\d t + 
    \int_0^{\tau}\int_{\{-\eta<u<0\}}\lvert\nabla u\rvert
    \lvert w+\nabla v\rvert\,\d x\,\d t \\
    &\leq \int_0^{\tau}\int_{\{-\eta<u<0\}}\lvert\partial_t u\rvert\,\d x\,\d t 
    + \frac{1}{\eta}\int_0^{\tau}\int_{\{-\eta<u<0\}}\lvert\nabla u\rvert^2\,\d x\,\d t 
    + \frac{\eta}{4}\int_0^{\tau}\int_{\{-\eta<u<0\}}\lvert w+\nabla v\rvert^2\,\d x\,\d t.
\end{align*}
Ultimately, we get
\[
    \int_{\Omega} u(\tau,x) \varphi_{\eta}(u(\tau,x)) \,\d x 
    \leq \int_0^{\tau}\int_{\{-\eta<u<0\}}\lvert\partial_t u\rvert\,\d x\,\d t 
    + \frac{\eta}{4}\int_0^{\tau}\int_{\{-\eta<u<0\}}\lvert w+\nabla v\rvert^2\,\d x\,\d t,
\]
which shows as $\eta \to 0$ that $u^-(\tau,\cdot)=0$ almost everywhere in $\Omega$.
\end{proof}

\section{Regularity and boundedness}
\label{Sec: Regularity and boundedness}
\noindent In order to establish the boundedness of solutions to~\eqref{Eq: KSNS} for small 
initial data, we change the functional setting from $X_0$ to a space $Y_0$. 
Let $q \in (\frac{n}{2} , \infty)$ satisfy~\eqref{Eq: Condition on Lebesgue exponent} and define
\[
    Y_0 := \L^q(\Omega)\times\L^{q^{*}}(\Omega)\times\L^q_{\sigma}(\Omega) \quad \text{and} \quad \cY_0 := \L^q_{\mathrm{av}}(\Omega)\times\L^{q^{*}}(\Omega)\times\L^q_{\sigma}(\Omega).
\]
Recall that, since $q > n / 2$, we have that $q^* > n$. As before, $\cA$ denotes the 
operator matrix defined in~\eqref{Eq: Operator matrix} but is now endowed with the domain
\[
    Y_1 := \dom(\Delta_q) \times \dom(\Delta_{q^*}) \times \dom(A) \quad \text{and} \quad \cY_1 := \dom(\Delta_q) \cap \L^q_{\mathrm{av}} (\Omega) \times \dom(\Delta_{q^*}) \times \dom(A).
\]
Moreover, for $1 <r<\infty$, we define the corresponding maximal regularity space by
\[
    \IE_{r,1}^{\infty}(\cY_0,\cY_1):= \W^{1,r} (0,\infty;\cY_0)\cap\L^r(0,\infty;\cY_1).
\]
The boundedness in space and time will follow by the trace theorem, whenever we know 
that the trace space $(\cY_0,\cY_1)_{1-1/r,r}$ continuously embeds into $\L^{\infty}$. 
This is guaranteed by the following result.

\begin{lemma}
\label{Lem: Embedding Linfty}
Let $q \in (\frac{n}{2} , \infty)$ satisfy~\eqref{Eq: Condition on Lebesgue exponent}. There exists $r_0 \in (1 , \infty)$ such that for all $r > r_0$ we have
\[
 \big(\cY_0,\cY_1\big)_{1-1/r,r} \hookrightarrow 
 \L^{\infty}(\Omega)\times\L^{\infty}(\Omega)\times\L^{\infty}(\Omega;\IR^n)
\]
and
\[
 \big(\cY_0,\cY_1\big)_{1-1/r,r} \hookrightarrow 
 \L^n(\Omega)\times\W^{1,n}(\Omega)\times\L^{\infty}(\Omega;\IR^n).
\]
\end{lemma}

\begin{proof}
For the first embedding, notice that, as before, we have
\begin{align}
\label{Eq: Domains in first-order Sobolev space}
 \dom(\Delta_q)\times\dom(\Delta_{q^*})\times\dom(A)\hookrightarrow 
 \W^{1,q^*}(\Omega)\times\W^{1,q^*}(\Omega)\times\W^{1,q^*}_{0,\sigma}(\Omega)=:Z.
\end{align}
This implies the embedding $\big(\cY_0,\cY_1\big)_{1-1/r,r} 
\hookrightarrow \big(\cY_0,Z\big)_{1-1/r,r}$ so that, since $q^*>n$, we can choose 
$r$ large enough such that the claim follows by Sobolev embedding. 

Concerning the second embedding, notice that
\[
 \dom(\Delta_{q^*}) \hookrightarrow \H^{1+\delta,q^*} (\Omega)
\]
for some $\delta>0$, see~\cite[Thm.~9.2]{Fabes_Mendez_Mitrea_1998}. Using the 
same argument as before for the first and third component, we thus arrive at
\[
 \big(\cY_0,\cY_1 \big)_{1-1/r,r}\hookrightarrow 
 \L^n(\Omega)\times\W^{1,n}(\Omega)\times\L^{\infty}(\Omega;\IR^n). \qedhere
\]
\end{proof}

The following lemma provides an appropriate bilinear estimate in this functional setting.

\begin{lemma}
\label{Lem: Bilinear estimates bounded}
Let $q \in (\frac{n}{2} , \infty)$ satisfy~\eqref{Eq: Condition on Lebesgue exponent}. There exists $r_0 \in (1 , \infty)$ such that for all $r \in (r_0 , \infty)$ 
there exists $C > 0$ such that for all 
$U_1 , U_2 \in \IE_{r , 1}^{\infty} (\cY_0 , \cY_1)$ we have
\[
    \| \Phi(U_1 , U_2) \|_{\L^r(0 , \infty ; \cY_0)} \leq 
    C \| U_1 \|_{\IE_{r , 1}^{\infty} (\cY_0 , \cY_1)} 
    \| U_2 \|_{\IE_{r , 1}^{\infty} (\cY_0 , \cY_1)}.
\]
\end{lemma}

\begin{proof}
The four terms are estimated as follows
\[
\|w_1 \cdot \nabla u_2\|_{\L^q(\Omega)} 
\leq \|w_1\|_{\L^n(\Omega)} \|\nabla u_2\|_{\L^{q^*}(\Omega)}
\leq C \|w_1\|_{\L^n(\Omega)} \|u_2\|_{\dom(\Delta_q)}.
\]
For the second one, we estimate
\begin{align*}
    \| \divergence(u_1 \nabla v_2)\|_{\L^q(\Omega)} 
    &\leq \|u_1 \Delta v_2\|_{\L^q (\Omega)} 
    + \|\nabla u_1 \cdot \nabla v_2\|_{\L^q (\Omega)} \\
    &\leq \|u_1\|_{\L^n(\Omega)} \| \Delta v_2 \|_{\L^{q^*}(\Omega)} 
    + \| \nabla u_1 \|_{\L^{q^*} (\Omega)} \|\nabla v_2\|_{\L^n(\Omega)} \\
    &\leq C \Big(\|u_1\|_{\L^n (\Omega)} \|v_2\|_{\dom(\Delta_{q^*})} 
    + \|u_1\|_{\dom(\Delta_q)} \|\nabla v_2\|_{\L^n(\Omega)}\Big).
\end{align*}
The third and fourth terms are treated similarly, resulting in the estimates
\[
    \|w_1 \cdot \nabla v_2\|_{\L^{q^*}(\Omega)} 
    \leq C \|w_1\|_{\L^{\infty}(\Omega)} \|v_2\|_{\dom(\Delta_{q^*})} 
\]
and
\[
    \|\IP[(w_1 \cdot \nabla)w_2 ]\|_{\L^q(\Omega)}
    \leq C \|w_1\|_{\L^n(\Omega)}\|w_2\|_{\dom(A)}.
\]
Summarizing we obtain the space estimate
\[
    \|\Phi(U_1,U_2)\|_{\cY_0} 
    \leq C \Big(\|U_1\|_{\L^n\times \W^{1,n}\times \L^{\infty}}\|U_2\|_{\cY_1} 
    + \|U_1\|_{\cY_1}\|U_2\|_{\L^n\times \W^{1,n} \times \L^{\infty}}\Big).
\]
We proceed with the estimates in time and use that
\[
    \IE_{r,1}^{\infty} (\cY_0,\cY_1) \hookrightarrow 
    \mathrm{BUC}([0,\infty);(\cY_0,\cY_1)_{1-\frac{1}{r},r})
\]
and, by virtue of Lemma~\ref{Lem: Embedding Linfty}, we choose $r_0$ large enough 
such that for $r > r_0$ we have
\[
    (\cY_0,\cY_1)_{1-\frac{1}{r},r} 
    \hookrightarrow \L^n(\Omega)\times \W^{1,n}(\Omega) \times \L^{\infty}(\Omega). \qedhere
\]
\end{proof}

\begin{lemma}
Let $r \in (1 , \infty)$ and $q \in (\frac{n}{2} , \infty)$ be subject to~\eqref{Eq: Condition on Lebesgue exponent}. There exists $\delta > 0$ such that for all $f \in \L^n (\Omega ; \IR^n)$ 
satisfying 
\begin{align*}
 \lvert u_s \rvert \|f\|_{\L^2(\Omega)}+\lvert u_s \rvert \leq \delta
\end{align*}
the operator $\cA+\cB_s$ has maximal $\L^p_{\mu}$-regularity on $(0,\infty)$ 
on the ground space $\cY_0$. In particular, $0 \in \rho(\cA+\cB_s)$.
\end{lemma}

\begin{proof}
We follow the proof of Lemma~\ref{Lem: Perturbation} and use the notation therein. The only difference is that the function space $\W^{-1 , q}_{\mathrm{av}} (\Omega)$ in the first component of $\cX_0$ is replaced by $\L^q_{\mathrm{av}} (\Omega)$. Thus, it is sufficient to replace the estimates for $w_s \cdot \nabla \widehat{u}$ and $u_s \Delta \widehat{v}$ by the following two estimates: By virtue~\eqref{Eq: Square roots}, a Sobolev-Poincar\'e inequality for mean-value free functions and Lemma~\ref{Eq: Solvability stationary Navier-Stokes}, we have  
\begin{align*}
    \|w_s \cdot \nabla \widehat{u}\|_{\L^q (\Omega)} &\leq \| w_s \|_{\L^n (\Omega)} \| \nabla \widehat{u} \|_{\L^{q^*} (\Omega)} \\
    &\leq C \| w_s \|_{\L^n (\Omega)} \| (- \Delta_{q^*})^{\frac{1}{2}} \widehat{u} \|_{\L^{q^*} (\Omega)} \\
    &\leq C \| w_s \|_{\L^n (\Omega)} \| \nabla (- \Delta_{q^*})^{\frac{1}{2}} \widehat{u} \|_{\L^q (\Omega)} \\
    &\leq C \delta (1 + \delta) \| \Delta_q \widehat{u}\|_{\L^q (\Omega)}. 
\end{align*}
Similarly, since $\Delta_{q^*} (\Id - \Delta_{q^*})^{-1}$ is bounded on $\L^{q^*} (\Omega) \hookrightarrow \L^q (\Omega)$, we have
\begin{align*}
    \|u_s \Delta \widehat{v}\|_{\L^q (\Omega)} 
    &\leq C \delta \| (\Id - \Delta_{q^{*}}) \widehat{v}\|_{\L^{q^*} (\Omega)} \leq C \delta \big( \| (\Id - \Delta_{q^{*}}) \widehat{v} - \widehat{u}\|_{\L^{q^*} (\Omega)} +   \| \Delta_q \widehat{u}\|_{\L^q (\Omega)} \big).
\end{align*}
Now, the proof is concluded as in Lemma~\ref{Lem: Perturbation}.
\end{proof}

\begin{proof}[Proof of Theorem~\ref{Thm: Boundedness and positivity}]
Mimicking the proof of Theorem~\ref{Thm: Existence/Stability}, but by using 
the bilinear estimate from Lemma~\ref{Lem: Bilinear estimates bounded} 
for $r > r_0$, there exists a stationary solution 
$U_s=(u_s,v_s,w_s)$ with
\[
 u_s=v_s=u_{0,\Omega}
\]
and $w_s\in \L^{\infty}(\Omega;\IR^n)$ (provided by 
Lemma~\ref{Eq: Solvability stationary Navier-Stokes}) and a global 
strong solution $(u,v,w)$ to~\eqref{Eq: KSNS} satisfying
\[
 \|t\mapsto \e^{\lambda t} [(u (t),v (t),w (t))-(u_s,v_s,w_s)]
 \|_{\IE_{r,1}^{\infty}(\cY_0,\cY_1)}\leq C
\]
for the given initial data $(u_0,v_0,w_0) \in (Y_0,Y_1)_{1-1/r,r}$. 
Choose $r\geq p$ for $p$ given in Theorem~\ref{Thm: Existence/Stability}. 
Then, for every $0<T<\infty$, we have
\[
 \IE_{r,1}^T(\cY_0,\cY_1)\hookrightarrow \IE_{p,\mu}^T(\cX_0,\cX_1).
\]
By uniqueness in $\IE_{p,\mu}^T(\cX_0,\cX_1)$, the constructed solution 
$(u,v,w)$ coincides with the solution provided by 
Theorem~\ref{Thm: Existence/Stability}. By Lemma~\ref{Lem: Embedding Linfty} 
and the trace theorem, we conclude that
\[
 (u,v,w) - (u_s , v_s , w_s) \in \L^{\infty}(0,\infty;\L^{\infty}(\Omega)\times\L^{\infty}(\Omega) 
 \times \L^{\infty}(\Omega ; \IR^n)).
\]
remains globally bounded in
space and time. Since $(u_s , v_s , w_s)$ is independent of $t$ and bounded in space, we conclude that $(u,v,w)$ is globally bounded in space and time. \par
For the rest of the proof, we assume that $u_0\geq 0$ and we show that $u$ 
remains positive. For this purpose, we may add another condition on $r$, 
namely, we assume that $r\geq 2$. For finite $T$, the property
\[
 u \in \W^{1,r}(0,T;\L^q(\Omega))
\]
implies that $u \in \W^{1,1}(0,T;\L^1(\Omega))$. 
Moreover,~\eqref{Eq: Domains in first-order Sobolev space} yields that
\[
Y_1\hookrightarrow \H^1(\Omega)\times\H^1(\Omega)\times\L^2_{\sigma}(\Omega)
\]
and thus
\[
 v\in\L^2(0,T;\H^1(\Omega))\quad\text{and}\quad w\in \L^2(0,T;\L^2_{\sigma}(\Omega)).
\]
The positivity of $u$ now follows from Proposition~\ref{Prop: Positivity}.
\end{proof}

\section{The Chemotaxis-Consumption Model}
\label{Sec: A variant of (KSNS)}

\noindent On a bounded Lipschitz domain $\Omega\subset\IR^n$, we study a variant of our 
initially considered Keller-Segel-Navier-Stokes system given 
by~\eqref{Eq: KSNS-variant}. As in Section~\ref{Sec: Preliminaries and Main Results}
we rewrite the system as an abstract semilinear system of the 
form~\eqref{Eq: abstract system} on the space
\[
 X_0:=\W^{-1,q}_0(\Omega)\times\L^{q^*}(\Omega)\times\L^q_{\sigma}(\Omega) 
 \quad \text{or on} \quad 
 \cX_0:=\W^{-1,q}_{\mathrm{av}}  (\Omega)
 \times\L^{q^{*}}_{\mathrm{av}}(\Omega)\times\L^q_{\sigma}(\Omega).
\]
Here, the operator $\cA$ is given by
\[
 \cA :=
    \begin{pmatrix}
     - \Delta & 0 & 0 \\ 
     0 & -\Delta_{q^*} & 0 \\
     - \IP[f \boldsymbol{\cdot}] & 0 & A
    \end{pmatrix}
\]
and equipped with the domain
\[
 X_1:=\W^{1,q}(\Omega)\times\dom(\Delta_{q^*})\times\dom(A) 
 \quad \text{or} \quad 
 \cX_1:=[\W^{1,q}(\Omega)\cap\L^q_{\mathrm{av}}(\Omega)] 
 \times [\dom(\Delta_{q^{*}}) \cap \L^{q^*}_{\mathrm{av}} (\Omega)]\times\dom(A).
\]
The nonlinearity is given by
\[
 \Phi (U_1,U_2) := \begin{pmatrix} 
    \Phi_1(U_1,U_2) \\ \Phi_2 (U_1,U_2) \\ \Phi_3(U_1,U_2) 
    \end{pmatrix} := \begin{pmatrix} 
    -w_1\cdot\nabla u_2-\divergence(u_1 \nabla v_2)\\
    -w_1 \cdot\nabla v_2-u_1 v_2\\
    -\IP(w_1\cdot\nabla)w_2 \end{pmatrix} 
    \cdotp
\]
Compared to the nonlinearity in~\eqref{eq:bilinearmap} we only need to estimate 
the term $-u_1 v_2$. Note that with the same notation as in 
Lemma~\ref{lem:bilinearmap} we have
\[
 \|u_1 v_2\|_{\L^{q^*}(\Omega)}
 \leq \|u_1\|_{\L^s(\Omega)}\|v_2\|_{\L^{r^*}(\Omega)} 
 \leq C \|u_1\|_{\L^s(\Omega)}\|v_2\|_{\W^{1,r}(\Omega)} 
 \leq \|U_1\|_{X_{\beta^{\prime}}}\|U_2\|_{X_{\beta}}.
\]
Thus, $\Phi$ satisfies the same estimates as in Lemma~\ref{lem:bilinearestimates}. 
Following the proof of Theorem~\ref{thm:localex} we find that for all initial data
\[
 (u_0,v_0,w_0) \in \B^{n/q-2}_{q,p,0}(\Omega)\times \B^{n/q^*}_{q^*,p}(\Omega)
 \times \B^{n/q-1}_{q,p,0,\sigma}(\Omega) 
\]
there exists a time $T > 0$ and a unique strong solution in 
$\IE_{p,\mu}^T (X_0,X_1)$. 

To construct global-in-time solutions we consider stationary solutions 
to~\eqref{Eq: KSNS-variant} of the form
\[
 U_s=(u_s,v_s,w_s) \quad \text{with} \quad u_s:=u_{0,\Omega}, \quad v_s=0
\]
and $w_s$ solves~\eqref{Eq: Stationary Navier-equations}. Mimicking the proof of 
Theorem~\ref{Thm: Existence/Stability} we find for all initial data close to the stationary solution $(u_s , 0 , w_s)$ in 
$\B^{n/q-2}_{q,p,0}(\Omega) \times \B^{n/q^*}_{q^*,p}(\Omega)
 \times \B^{n/q-1}_{q,p,0,\sigma}(\Omega)$ 
 a unique global strong solution to~\eqref{Eq: KSNS-variant} satisfying
\[
 \e^{\lambda t}
 \bigl\|(u (t),v (t),w (t))-(u_s,v_s,w_s)\bigr\|_{(\cX_0,\cX_1)_{\mu-1/p,p}} 
 \xrightarrow[t\to\infty]{} 0
\]
for some $\lambda > 0$. 

As in Section~\ref{Sec: Regularity and boundedness} the solution is globally 
bounded in space and time whenever the initial data in $(Y_0,Y_1)_{1-1/r,r}$ are sufficiently close to the stationary solution 
in $(\cY_0,\cY_1)_{1-1/r,r}$, where
\[
    Y_0:=\L^q(\Omega)\times\L^{q^{*}}(\Omega)\times\L^q_{\sigma}(\Omega) 
    \quad \text{and} \quad 
    \cY_0 := \L^q_{\mathrm{av}}(\Omega)\times\L^{q^{*}}_{\mathrm{av}} (\Omega)\times\L^q_{\sigma}(\Omega)
\]
and
\[
    Y_1:=\dom(\Delta_q)\times \dom(\Delta_{q^*})\times \dom(A) 
    \quad \text{and} \quad 
    \cY_1 := [\dom(\Delta_q) \cap \L^q_{\mathrm{av}}(\Omega)] \times [\dom(\Delta_{q^*}) \cap \L^{q^{*}}_{\mathrm{av}} (\Omega)] \times \dom(A)
\]
and $r\geq 2$ is sufficiently large. Since the first equation 
in~\eqref{Eq: KSNS-variant} is unchanged compared to~\eqref{Eq: KSNS}, 
the solution is positive, provided the initial datum $u_0$ is positive. 

Summing up, we arrive at the following theorem.

\begin{theorem}
\label{Thm: Variant}
Let all parameters be chosen as in Theorems~\ref{thm:localex}, \ref{Ass: Forcing},
\ref{Thm: Existence/Stability} and~\ref{Thm: Boundedness and positivity}. 
Then the following statements about~\eqref{Eq: KSNS-variant} are valid:
\begin{enumerate}
 \item (Local solvability) 
 For all initial data
 \[
(u_0,v_0,w_0) \in \B^{n/q-2}_{q,p,0}(\Omega)\times\B^{n/q^*}_{q^*,p}(\Omega)
 \times \B^{n/q-1}_{q,p,0,\sigma}(\Omega) 
\]
and all $f \in \L^n (\Omega ; \IR^n)$ there exists $T>0$ such that the Keller-Segel-Navier-Stokes system 
\eqref{Eq: KSNS-variant} admits a unique strong solution 
\begin{align*}
 (u,v,w) \in \IE_{p,\mu}^T(X_0,X_1), \quad \text{where} \quad \mu= \frac{n}{2q} + \frac{1}{p} - \frac{1}{2}.
\end{align*}
 \item \label{Item: Global/Stability} (Global solvability and stability) 
 There exist $\delta,\lambda,C>0$ such that for all $0 < \kappa < C$ and all
\[
(u_0,v_0,w_0) \in \B^{n/q-2}_{q,p,0}(\Omega)\times \B^{n/q^*}_{q^*,p}(\Omega)
 \times \B^{n/q-1}_{q,p,0,\sigma}(\Omega) \quad \text{with} \quad v_{0 , \Omega} = 0 
\]
and
\[
 \|u_0 - u_s\|_{\B^{n/q-2}_{q,p,0}(\Omega)}+\|v_0\|_{\B^{n/q^*}_{q^*,p}(\Omega)} 
 +\|w_0 - w_s\|_{\B^{n/q-1}_{q,p,0,\sigma}(\Omega)} \leq \kappa
\]
and all $f \in \L^n (\Omega ; \IR^n)$ satisfying Assumption~\ref{Ass: Forcing} 
for this $\delta$, there exists a unique global strong solution $(u,v,w)$ 
to~\eqref{Eq: KSNS-variant} satisfying
\[
 \|t\mapsto \e^{\lambda t}[(u(t),v(t),w(t))-
 (u_s,v_s,w_s)]\|_{\IE_{p,\mu}^{\infty}(\cX_0,\cX_1)}\leq 2\kappa,
\]
where $\mu=n/(2q)+1/p-1/2$. In particular,
\[
 \e^{\lambda t}\bigl\|(u(t),v(t),w(t))-
 (u_s, 0 ,w_s)\bigr\|_{(\cX_0,\cX_1)_{\mu-1/p,p}} \xrightarrow[t\to \infty]{} 0.
\]
 \item (Positivity and boundedness) 
 There exist $\delta,\lambda,C>0$ and $r_0>1$ such that for all $0 < \kappa < C$, all $r>r_0$ and all
\[
 (u_0,v_0,w_0)\in\big(Y_0,Y_1\big)_{1-1/r,r} 
 \quad \text{with} \quad v_{0 , \Omega} = 0
\]
and
\[
\|(u_0 - u_s,v_0,w_0 - w_s) \|_{(\cY_0,\cY_1)_{1-1/r,r}} \leq \kappa
\]
the unique solution constructed in~\eqref{Item: Global/Stability} satisfies
\[
 \|t\mapsto \e^{\lambda t}[(u (t),v(t),w(t))-
 (u_s, 0 ,w_s) ]\|_{\IE_{r,1}^{\infty}(\cY_0,\cY_1)} \leq 2\kappa.
\]
Moreover, it is globally bounded in space and time, i.e.,
\[
 (u,v,w)\in\L^{\infty}(0,\infty;\L^{\infty}(\Omega)\times\L^{\infty}(\Omega) 
 \times \L^{\infty}(\Omega;\IR^n)),
\]
and, if $u_0 \geq 0$ a.e., then $u \geq 0$ for all a.e.\@ $t>0$ and a.e.\@ $x\in\Omega$.
\end{enumerate}
\end{theorem}

\section{Appendix: Real interpolation of solenoidal function spaces}

\noindent We will briefly present the real interpolation result of solenoidal function spaces needed in the proofs of Theorems~\ref{thm:localex} and~\ref{Thm: Existence/Stability}. For this purpose, we introduce the scales of Bessel potential and Besov spaces in more detail. \par
Let $\Omega \subset \IR^n$, $n \in \IN$ with $n \geq 2$, be a bounded Lipschitz domain. Following~\cite[Sec.~2]{Mitrea_Monniaux} we introduce for $1 < q < \infty$ and $s > \frac{1}{q} - 1$ the Bessel potential space of tempered distributions on $\IR^n$ encoding a ``zero trace'' on $\partial \Omega$ by
\begin{align*}
 \H^{s , q}_0 (\Omega ; \IC^n) := \{ w \in \H^{s , q} (\IR^n ; \IC^n) : \supp(w) \subset \overline{\Omega} \}
\end{align*}
and endow it with the inherited norm of $\H^{s , q} (\IR^n ; \IC^n)$, where $\H^{s , q} (\IR^n ; \IC^n)$ denotes the Bessel potential space on $\IR^n$. Moreover, we define
\begin{align*}
 \H^{s , q}_z (\Omega ; \IC^n) := \{ w \in \cD^{\prime} (\Omega ; \IC^n) : \exists W \in \H^{s , q}_0 (\Omega ; \IC^n) \text{ with } W|_{\Omega} = w \}
\end{align*}
and
\begin{align*}
 \H^{s , q} (\Omega ; \IC^n) := \{ w \in \cD^{\prime} (\Omega ; \IC^n) : \exists W \in \H^{s , q} (\IR^n ; \IC^n) \text{ with } W|_{\Omega} = w \}
\end{align*}
both endowed with the natural infimum norm. Note that $\H^{s , q}_0 (\Omega ; \IC^n)$ and $\H^{s , q}_z (\Omega ; \IC^n)$ can be identified via restriction to $\Omega$, or conversely, via extension to $\IR^n$ by zero, see~\cite[p.~1530]{Mitrea_Monniaux}. \par 
For the description of solenoidal function spaces the normal trace of a vector field is crucial. It is defined as follows: If $\frac{1}{q} - 1 < s < \frac{1}{q}$, we define the mapping
\begin{align*}
 \nu \, \cdot \, &: \big\{ w \in \H^{s , q} (\Omega ; \IC^n) : \divergence(w) \in \H^{s , q} (\Omega ; \IC^n) \big\} \to \B^{s - 1 / q}_{q , q} (\partial \Omega) \\
 \nu \cdot w &:= \Big[\phi \mapsto \langle \Phi , \divergence (w) \rangle_{\H^{- s , q^{\prime}} , \H^{s , q}} + \langle \nabla \Phi , w \rangle_{\H^{- s , q^{\prime}} , \H^{s , q}} \Big],
\end{align*}
where $\phi \in \B^{-(s - 1/q)}_{q^{\prime} , q^{\prime}} (\partial \Omega)$ and $\Phi \in \H^{1 - s , q^{\prime}} (\Omega)$ is such that $\Phi|_{\partial \Omega} = \phi$. The well-posedness of this definition was established in~\cite[Prop.~2.7]{Mitrea_Monniaux}. \par
If $s \geq \frac{1}{q}$ and $w \in \H^{s , q} (\Omega ; \IC^n)$ is such that $\divergence(w) \in \H^{s , q} (\Omega ; \IC^n)$, then $\nu \cdot w$ is interpreted as an element of
\begin{align*}
 \nu \cdot w \in \bigcap_{\frac{1}{q} - 1 < \alpha < \frac{1}{q}} \B^{\alpha - 1 / q}_{q , q} (\partial \Omega).
\end{align*}
In~\cite[Prop.~2.10]{Mitrea_Monniaux}, Mitrea and Monniaux proved for all $1 < q < \infty$ and $s > \frac{1}{q} - 1$ that
\begin{align*}
 \overline{\C_{c , \sigma}^{\infty} (\Omega)}^{\H^{s , q}_z} = \big\{ w \in \H^{s , q}_z (\Omega ; \IC^d) : \divergence (w) = 0 \text{ and } \nu \cdot w = 0 \text{ on } \partial \Omega \big\} =: \H^{s , q}_{0 , \sigma} (\Omega),
\end{align*}
where $\C_{c , \sigma}^{\infty} (\Omega)$ is the space of all smooth and solenoidal vector fields with compact support in $\Omega$. \par 
Analogously as above, we define for $1 < p , q < \infty$ and $s > \frac{1}{q} - 1$
\begin{align*}
 \B^s_{q , p ,0} (\Omega ; \IC^n) := \{ w \in \B^s_{q , p} (\IR^n ; \IC^n) : \supp(w) \subset \overline{\Omega} \}
\end{align*}
and endow it with the inherited norm of $\B^s_{q , p} (\IR^n ; \IC^n)$, where $\B^s_{q , p} (\IR^n ; \IC^n)$ denotes the Besov space on $\IR^n$, and we define
\begin{align*}
 \B^s_{q , p , z} (\Omega ; \IC^n) := \{ w \in \cD^{\prime} (\Omega ; \IC^n) : \exists W \in \B^s_{q , p , 0} (\Omega ; \IC^n) \text{ with } W|_{\Omega} = w \}
\end{align*}
and
\begin{align*}
 \B^s_{q , p} (\Omega ; \IC^n) := \{ w \in \cD^{\prime} (\Omega ; \IC^n) : \exists W \in \B^s_{q , p} (\IR^n ; \IC^n) \text{ with } W|_{\Omega} = w \}
\end{align*}
and endow both with the natural infimum norm. As before, one can identify $\B^s_{q , p ,0} (\Omega ; \IC^n)$ and $\B^s_{q , p , z} (\Omega ; \IC^n)$ via restriction to $\Omega$, or conversely, via extension to $\IR^n$ by zero. Moreover,~\cite[Thm.~3.5]{Triebel_Lipschitz} guarantees that
\begin{align*}
 \B^{s}_{q , p , 0} (\Omega ; \IC^n) = (\B^{-s}_{q^{\prime} , p^{\prime}} (\Omega ; \IC^n))^* \quad \text{for} \quad \frac{1}{q} + \frac{1}{q^{\prime}} = 1 = \frac{1}{p} + \frac{1}{p^{\prime}} \quad \text{and} \quad 0 \geq s > \frac{1}{q} - 1.
\end{align*}
Thus, for $s \leq \frac{1}{q} - 1$ we extend this scale by defining
\begin{align*}
 \B^{s}_{q , p , 0} (\Omega ; \IC^n) := (\B^{- s}_{q^{\prime} , p^{\prime}} (\Omega ; \IC^n))^*.
\end{align*}
If $s > \frac{1}{q} - 1$ the normal trace of an element in $\B^s_{q , p} (\Omega ; \IC^n)$ may be defined via the embedding $\B^s_{q , p} (\Omega ; \IC^n) \subset \H^{s^{\prime} , q} (\Omega ; \IC^n)$, $s^{\prime} < s$, since for $w \in \B^s_{q , p} (\Omega ; \IC^n)$ we can interprete $w \cdot \nu$ as an element of
\begin{align*}
 \nu \cdot w \in \bigcap_{\frac{1}{q} - 1 < s^{\prime} < \min(s , \frac{1}{q})} \B^{s^{\prime} - 1 / q}_{q , q} (\partial \Omega).
\end{align*}
Finally, we define
\begin{align*}
 \B^s_{q , p , 0 , \sigma} (\Omega) := \{ w \in \B^s_{q , p , z} (\Omega ; \IC^n) : \divergence(w) = 0 \text{ and } \nu \cdot w = 0 \text{ on } \partial \Omega \}.
\end{align*}
Our goal is now to establish the following interpolation result.

\begin{proposition}
\label{Prop: Real interpolation}
Let $\Omega \subset \IR^n$, $n \in \IN$ with $n \geq 2$, be a bounded Lipschitz domain, $s_0 , s_1 > \frac{1}{q} - 1$, $1 < p , q < \infty$ and $\theta \in (0 , 1)$. If $s := (1 - \theta) s_0 + \theta s_1$, then we have
\begin{align*}
 \big( \H^{s_0 , q}_{0 , \sigma} (\Omega) , \H^{s_1 , q}_{0 , \sigma} (\Omega) \big)_{\theta , p} = \B^s_{q , p , 0 , \sigma} (\Omega)
\end{align*}
with equivalent norms.
\end{proposition}

\begin{proof}
We closely follow the lines of the corresponding result for the complex interpolation method in~\cite[Thm.~2.12]{Mitrea_Monniaux}. We introduce the spaces
\begin{align*}
 X_j := \H^{s_j , q}_0 (\Omega ; \IC^n) \quad \text{and} \quad Y_j := \big\{ f \in \H^{s_j - 1 , q}_0 (\Omega) : \langle f , 1 \rangle_{\cE^{\prime} , \cE} = 0 \big\} \qquad (j = 0 , 1)
\end{align*}
where $\cE$ denotes the space of smooth functions and $\cE^{\prime}$ the space of distributions with compact support, and define
\begin{align*}
 D : X_j \to Y_j, \; w \mapsto \divergence(w) \qquad (j = 0 , 1).
\end{align*}
Moreover, for $j = 0 , 1$, we define
\begin{align*}
 G : \H^{s_j - 1 , q_j}_0 (\Omega) \to X_j, \; f \mapsto \widetilde{(K \circ \psi)(f)},
\end{align*}
where the tilde means ``entension to $\IR^n$ by zero'', $K : (\H^{s_j - 1 , q_j^{\prime}} (\Omega))^{*} \to \H^{s_j , q_j}_z (\Omega ; \IC^n)$ denotes the Bogovsk\u{\i}i operator studied in~\cite[Prop.~2.5]{Mitrea_Monniaux} and $\psi$ denotes the extension of the operator
\begin{align*}
 \widetilde{\C_c^{\infty} (\Omega)} \to \C_c^{\infty} (\Omega), \; \widetilde{f} \mapsto f \quad \text{as an isomorphism} \quad \H^{1 - s_j , q_j}_0 (\Omega) \to (\H^{s_j - 1 , q_j^{\prime}} (\Omega))^{*}
\end{align*}
and was studied in~\cite[Lem.~2.4]{Mitrea_Monniaux}. It was then proven in~\cite[p.~1541]{Mitrea_Monniaux} that
\begin{align*}
 D \circ G = \Id \text{ on } Y_j \text{ for } j = 0 , 1.
\end{align*}
The previous relation implies that $\cP := G \circ D$ restricts to bounded linear projections on $X_j$ for each $j = 0 , 1$. Consequently, the kernel of $\cP|_{X_j}$ is a complemented subspace of $X_j$ and it is given by
\begin{align*}
 \ker (\cP|_{X_j}) = \ker(D|_{X_j}) = \big\{ w \in \H^{s_j , q}_0 (\Omega ; \IC^n) : \divergence(w) = 0 \big\}.
\end{align*}
Interpolation of complemented subspaces~\cite[Thm.~1.17.1]{Triebel} now yields that
\begin{align*}
 \big( \ker(\cP|_{X_0}) , \ker(\cP|_{X_1}) \big)_{\theta , p} = \big\{ w \in \big( \H^{s_0 , q}_0 (\Omega ; \IC^n) , \H^{s_0 , q}_0 (\Omega ; \IC^n) \big)_{\theta , p} : \divergence(w) = 0 \big\}.
\end{align*}
Next, it was proven below~\cite[Eq.~(2.106)]{Mitrea_Monniaux} that the restriction operator
\begin{align*}
 \boldsymbol{\cdot}\,|_{\Omega} : \ker(\cP|_{X_0}) \to \H^{s_j , q}_{0 , \sigma} (\Omega) \qquad (j = 0 , 1)
\end{align*}
is an isomorphism. Hence,
\begin{align*}
 \big( \H^{s_0 , q}_{0 , \sigma} (\Omega) , \H^{s_1 , q}_{0 , \sigma} (\Omega) \big)_{\theta , p} = \big\{ w|_{\Omega} : w \in \big( \H^{s_0 , q}_0 (\Omega ; \IC^n) , \H^{s_0 , q}_0 (\Omega ; \IC^n) \big)_{\theta , p} \text{ and } \divergence(w) = 0 \big\}.
\end{align*}
Using~\cite[Thm.~3.5]{Triebel_Lipschitz} we readily find that
\begin{align*}
 \big( \H^{s_0 , q}_0 (\Omega ; \IC^n) , \H^{s_0 , q}_0 (\Omega ; \IC^n) \big)_{\theta , p} = \B^s_{q , p , 0} (\Omega ; \IC^n).
\end{align*}
Below~\cite[Eq.~(2.106)]{Mitrea_Monniaux} it was, in particular, proven that $w|_{\Omega} \cdot \nu = 0$ for all $w \in \H^{s^{\prime} - 1 / q , q}_0 (\Omega ; \IC^n)$ with $\divergence (w) = 0$, where $s^{\prime} > \frac{1}{q} - 1$, which implies that $w|_{\Omega} \cdot \nu = 0$ for all $w \in \B^s_{q , p , 0} (\Omega ; \IC^n)$ with $\divergence(w) = 0$. It follows that
\begin{align*}
 \big( \H^{s_0 , q}_{0 , \sigma} (\Omega) , \H^{s_1 , q}_{0 , \sigma} (\Omega) \big)_{\theta , p} = \big\{ w|_{\Omega} : w \in \B^s_{q , p , 0} (\Omega ; \IC^n) \text{ and } \divergence(w) = 0 \big\} = \B^s_{q , p , 0 , \sigma} (\Omega). &\qedhere
\end{align*}
\end{proof}

\noindent
{\bf Acknowledgements}

\noindent
It is a pleasure for us to thank the Centre International de Rencontres Math\'ematiques (CIRM) in Luminy, France, for their warm hospitality and for providing an excellent environment, in which the content of this article was developed. 

\noindent
We also would like to thank Michael Winkler for inspiring  discussions about the Keller-Segel system and its variants at CIRM.

\end{document}